\documentclass[11pt]{amsart}
\usepackage{amssymb}
\usepackage{amsthm}
\usepackage{amsmath}
\usepackage{latexsym}
\usepackage{comment}
\usepackage{hyperref}
\usepackage{verbatim}
\usepackage{xypic}
\usepackage{graphicx}
\usepackage[utf8]{inputenc}
\usepackage[margin=1in]{geometry}
\usepackage{fancyhdr}
\newtheorem{thm}{Theorem}[section]
\newtheorem{prop}[thm]{Proposition}

\newtheorem{lemma}[thm]{Lemma}
\newtheorem{cor}[thm]{Corollary}

\theoremstyle{definition} \newtheorem{ex}[thm]{Example}

\newtheorem{hyp}[thm]{Hypothesis}
\theoremstyle{definition} \newtheorem{rmk}[thm]{Remark}
\newcommand{\cc}{\mathbb{C}}
\newcommand{\rr}{\mathbb{R}}
\newcommand{\qq}{\mathbb{Q}}
\newcommand{\oper}[1]{\operatorname{#1}} 
\newcommand{\zz}{\mathbb{Z}}
\newcommand{\ff}{\mathbb{F}}
\newcommand{\aff}{\mathbb{A}}
\newcommand{\proj}{\mathbb{P}}

\newcommand{\Aut}{\mathrm{Aut}}
\newcommand{\Out}{\mathrm{Out}}
\newcommand{\Spec}{\mathrm{Spec}}

\newcommand{\id}{\mathrm{id}}
\newcommand{\im}{\mathrm{im}}

\newcommand{\et}{\mathrm{\acute{e}t}}

\newcommand{\tp}{\mathrm{top}}
\newcommand{\alg}{\mathrm{alg}}
\newcommand{\unr}{\mathrm{un}}
\newcommand{\an}{\mathrm{an}}

\graphicspath{{pics/}}

\title{Prime-to-$p$ \'{e}tale fundamental groups of punctured projective lines over strictly Henselian fields}
\author[Hasson]{Hilaf Hasson}
  \address{Hilaf Hasson: University of Maryland, College Park, MD 20742, USA}
  \email{hilaf@math.umd.edu}
  \author[Yelton]{Jeffrey Yelton}
  \address{Jeffrey Yelton: Emory University, Atlanta, GA 30322, USA}
  \email{jeffrey.yelton@emory.edu}

\begin{document}

\maketitle

\begin{abstract}

Let $K$ be the fraction field of a strictly Henselian DVR of residue characteristic $p \geq 0$ with algebraic closure $\bar{K}$, and let $\alpha_{1}, ... , \alpha_{d} \in \proj_{K}^{1}(K)$. In this paper, we give explicit generators and relations for the prime-to-$p$ \'etale fundamental group of $\proj_K^1\smallsetminus\{\alpha_1,...,\alpha_d\}$ that depend (solely) on their intersection behavior. This is done by a comparison theorem that relates this situation to a topological one. Namely, let $a_{1}, ... , a_{d}$ be distinct power series in $\cc[[x]]$ with the same intersection behavior as the $\alpha_i$'s, converging on an open disk centered at $0$, and choose a point $z_{0} \neq 0$ lying in this open disk. We compare the natural action of $\oper{Gal}(\bar{K} / K)$ on the prime-to-$p$ \'{e}tale fundamental group of $\proj_{\bar{K}} \smallsetminus \{\alpha_{1}, ... , \alpha_{d}\}$ to the topological action of looping $z_0$ around the origin on the fundamental group of $\proj_{\mathbb{C}}^1\smallsetminus\{a_1(z_0),...,a_d(z_0)\}$. This latter action is, in turn, interpreted in terms of Dehn twists. A corollary of this result is that every prime-to-$p$ $G$-Galois cover of $\proj_{\bar K}^1\smallsetminus\{\alpha_1,...,\alpha_d\}$ satisfies that its field of moduli (as a $G$-Galois cover) has degree over $K$  dividing the exponent of $G / Z(G)$.

\end{abstract}

\section{Introduction} \label{sec1}

Let $R$ be any strictly Henselian discrete valuation ring of characteristic $0$ with residue characteristic $p \geq 0$ and uniformizer $\pi$.  Let $K$ be the fraction field of $R$, and fix an algebraic closure $\bar{K}$ of $K$.  Let $\alpha_{1}, ... , \alpha_{d}$ be $K$-points of $\proj_{K}^{1}$ for some $d \geq 2$, and let $P$ be a geometric point lying over a $K$-point of $\proj_{K}^{1}$ different from $\alpha_{1}, ... , \alpha_{d}$.    We have the fundamental short exact sequence 
\begin{equation} \label{short exact sequence of algebraic fundamental groups} 1 \to \pi_{1}^{\et}(\proj_{\bar{K}} \smallsetminus \{\alpha_{1}, ... , \alpha_{d}\}, P) \to \pi_{1}^{\et}(\proj_{K}^{1} \smallsetminus \{\alpha_{1}, ... , \alpha_{d}\}, P) \to G_{K} \to 1. \end{equation}
The $K$-point $P$ corresponds to a section $\Spec(K) \to \proj_{K}^{1} \smallsetminus \{\alpha_{1}, ... , \alpha_{d}\}$, which induces a splitting of the short exact sequence (\ref{short exact sequence of algebraic fundamental groups}).  This gives rise to an algebraic monodromy action of $G_{K}$ on $\pi_{1}^{\et}(\proj_{\bar{K}}^{1} \smallsetminus \{\alpha_{1}, ... , \alpha_{d}\}, P)$, which we denote by $\rho_{\alg} : G_{K} \to \Aut(\pi_{1}^{\et}(\proj_{\bar{K}}^{1} \smallsetminus \{\alpha_{1}, ... , \alpha_{d}\}, P))$.  We write $\bar{\rho}_{\alg} : G_{K} \to \Out(\pi_{1}^{\et}(\proj_{\bar{K}}^{1} \smallsetminus \{\alpha_{1}, ... , \alpha_{d}\}, P))$ for the induced outer action.

For any group $G$, we write $\widehat{G}$ for its profinite completion.  For any profinite group $G$, let $G^{(p')}$ denote the maximal prime-to-$p$ quotient of $G$ if $p > 0$, and let $G^{(p')} = G$ if $p = 0$.  Since $G^{(p')}$ is a characteristic quotient of $G$, any action on $G$ induces an action on $G^{(p')}$. Let $\rho_{\alg}^{(p')}:G_K\rightarrow \oper{Aut}(\pi_{1}^{\et}(\proj_{\bar{K}}^{1} \smallsetminus \{\alpha_{1}, ... , \alpha_{d}\}, P))^{(p')})$ denote the action induced by $\rho_{\alg}$, and let $\bar \rho_{\alg}^{(p')}$ be the induced outer action. While a complete description of $\bar \rho_{\alg}$ remains out of reach, the main theorem of this paper (Theorem \ref{thm comparison}), together with Theorem \ref{thm main topological}, gives a complete and explicit description of $\bar \rho_{\alg}^{(p')}$.

\subsection{Comparison of topological and algebraic monodromy actions} \label{sec1.1}

Fix an integer $d \geq 2$, and choose distinct power series $a_{1}, ... , a_{d} \in \cc[[x]]$ with positive radii of convergence.  For each $i$ and any complex number $z \in \cc$, we write $a_{i}(z)$ for the evaluation of $a_{i}$ at $z$; if $a_{i}$ diverges at $z$, then we consider $a_{i}(z)$ to be the infinity point of $\proj_{\cc}^{1}$.  For any real number $\varepsilon > 0$, we write $B_{\varepsilon} := \{z \in \cc \ | \ |z| < \varepsilon\}$ for the (open) ball of radius $\varepsilon$ and let $B_{\varepsilon}^{*} = B_{\varepsilon} \smallsetminus \{0\}$.  We will always make the following assumption.

\begin{hyp} \label{hypothesis varepsilon}

The real number $\varepsilon > 0$ is small enough that each power series $a_{i}$ converges on the closure of $B_{\varepsilon}$, and for any $z \in B_{\varepsilon}^{*}$, we have $a_{i}(z) \neq a_{j}(z)$ for $i \neq j$.

\end{hyp}

We define a family $\mathcal{F} \to B_{\varepsilon}^{*}$ of punctured Riemann spheres by letting 
$$\mathcal{F} = ((\proj_{\cc}^{1})^{\an} \times B_{\varepsilon}^{*}) \smallsetminus \bigcup_{i = 1}^{d} \{(a_{i}(z), z) \ | \ z \in 
B_{\varepsilon}^{*}\}.$$
It follows from Hypothesis \ref{hypothesis varepsilon} that the obvious projection map $\mathcal{F} \to B_{\varepsilon}^{*}$ is a fiber bundle.  For each $z \in B_{\varepsilon}^{*}$, write $\infty_{z}$ for the point in the fiber $\mathcal{F}_{z} := (\proj_{\cc}^{1})^{\an} \smallsetminus \{a_{1}(z), ... , a_{d}(z)\}$.  Since the punctured disk $B_{\varepsilon}^{*}$ has trivial $j^{\oper{th}}$ homotopy group for $j \geq 2$, the associated long exact homotopy sequence of this fiber bundle truncates.  Therefore, after choosing a basepoint $z_{0} \in B_{\varepsilon}^{*}$, we have the short exact sequence of fundamental groups 
\begin{equation} \label{short exact sequence of topological fundamental groups} 1 \to \pi_{1}(\mathcal{F}_{z_{0}}, \infty_{z_0}) \to \pi_{1}(\mathcal{F}, \infty_{z_{0}}) \to \pi_{1}(B_{\varepsilon}^{*}, z_{0}) \to 1. \end{equation}

There is an obvious ``infinity section" of the family $\mathcal{F} \to B_{\varepsilon}^{*}$ given by $z \mapsto \infty_{z}$.  This section induces a splitting of the short exact sequence (\ref{short exact sequence of topological fundamental groups}).  This gives rise to a monodromy action of $\pi_{1}(B_{\varepsilon}^{*}, z_{0})$ on $\pi_{1}(\mathcal{F}_{z_{0}}, \infty_{z_0})$, which we denote by $\rho_{\tp} : \pi_{1}(B_{\varepsilon}^{*}, z_{0}) \to \Aut(\pi_{1}(\mathcal{F}_{z_{0}}, \infty_{z_0}))$.

We remark that taking profinite completions of the terms in the sequence (\ref{short exact sequence of algebraic fundamental groups}) produces a short exact sequence. Indeed, the profinite completion functor is right exact (\cite[Proposition 3.2.5]{zalesskii}); left-exactness of the resulting sequence now follows from \cite[Proposition 4]{leftexact} after observing that the groups appearing in (\ref{short exact sequence of algebraic fundamental groups}) are finitely generated, the sequence is split, and $\pi_{1}(B_{\varepsilon}^{*}, z_{0})$ is residually finite. We further remark that the natural maps from $\pi_{1}(B_{\varepsilon}^{*}, z_{0})$ and $\pi_{1}(\mathcal{F}_{z_{0}}, \infty_{z_0})$ to their respective profinite completions are embeddings, as these are free groups and therefore residually finite. By the Five Lemma, it follows that the natural map from $\pi_{1}(\mathcal{F}, \infty_{z_{0}})$ to its profinite completion is also injective. It follows that  $\rho_{\tp}$ extends uniquely to a continuous action of $\widehat{\pi}_{1}(B_{\varepsilon}^{*}, z_{0})$ on $\widehat\pi_{1}(\mathcal{F}_{z_{0}}, \infty_{z_0})$; we denote this action also by $\rho_{\tp}$.  We write $\bar{\rho}_{\tp} : \widehat{\pi}_{1}(B_{\varepsilon}^{*}, z_{0}) \to \Out(\widehat \pi_{1}(\mathcal{F}_{z_{0}}, \infty_{z_0}))$ for the induced outer monodromy action, and we write $\rho_{\tp}^{(p')}$ (resp. $\bar{\rho}_{\tp}^{(p')}$) for the action (resp. the outer action) induced by $\rho_{\tp}$ on $\widehat \pi_{1}(\mathcal{F}_{z_{0}}, \infty_{z_{0}})^{(p')}$. 
For $1 \leq i < j \leq d$, we write $e_{i, j} = v_{x}(a_i-a_j)$, where $v_x$ is the $x$-adic valuation, which is the intersection multiplicity of the power series $a_{i}$ and $a_{j}$ viewed as divisors on the surface $\proj_{\cc[[x]]}^{1}$.  Meanwhile, since $\proj_{R}^{1}$ is proper over $\Spec(R)$, the $K$-points $\alpha_{i} \in \proj_{K}^{1}$ extend to $R$-points of $\proj_{R}^{1}$, which we also denote by $\alpha_{i}$.  For $1 \leq i < j \leq d$, we write $E_{i, j}$ for the intersection multiplicity of the $R$-points $\alpha_{i}$ and $\alpha_{j}$ as divisors on the surface $\proj_{R}^{1}$.

We may visualize $\Spec(R)$ as an infinitesimally small disk and $\Spec(K)$ as an infinitesimally small punctured disk.  Our main theorem asserts that the action of $G_{K}$ on $\pi_{1}^{\et}(\proj_{\bar{K}}^{1} \smallsetminus \{\alpha_{1}, ... , \alpha_{d}\}, P)^{(p')}$ agrees with this intuition: namely, it is ``the same" as the monodromy action of $\pi_{1}(B_{\varepsilon}^{*}, z_{0})$ on $\pi_{1}(\mathcal{F}_{z_{0}}, \infty_{z_0})^{(p')}$ as long as the intersection behavior of the $\alpha_{i}$'s over $\Spec(R)$ is the same as the intersection behavior of the $a_{i}$'s over $B_{\varepsilon}$.

\begin{thm} \label{thm comparison}

In the above situation, we have the following.

a) The actions $\rho_{\tp}^{(p')}$ and $\rho_{\alg}^{(p')}$ factor through $\pi_{1}(B_{\varepsilon}^{*}, z_{0})^{(p')}$ and $G_{K}^{(p')}$ respectively.

b) If $e_{i, j} = E_{i, j}$ for $1 \leq i < j \leq d$, then there exist isomorphisms $\widehat{\pi}_{1}(B_{\varepsilon}^{*}, z_{0})^{(p')} \stackrel{\sim}{\to}G_{K}^{(p')}$ and 
$\widehat\pi_{1}(\mathcal{F}_{z_{0}}, \infty_{z_0})^{(p')} \stackrel{\sim}{\to} \pi_{1}^{\et}(\proj_{\bar{K}}^{1} \smallsetminus 
\{\alpha_{1}, ... , \alpha_{d}\}, P)^{(p')}$ inducing an isomorphism of the outer actions $\bar{\rho}_{\tp}^{(p')}$ and $\bar{\rho}_{\alg}^{(p')}$.  Moreover, this lifts to an isomorphism of the actions $\rho_{\tp}^{(p')}$ and $\rho_{\alg}^{(p')}$ provided that $\alpha_i \in R \cup \{\infty\}$ for $1 \leq i \leq d$.

\end{thm}
\begin{rmk} \label{rmk applications}
In \cite{yelton2017boundedness}, the second author derives a particular application of the above theorem towards understanding the $\ell$-adic Galois action action associated to the Jacobians of hyperelliptic curves over number fields, through viewing such a curve as a cover of the projective line ramified at certain points and localizing at primes $p \neq \ell, 2$ at which the curve has bad reduction.

\end{rmk}

\begin{rmk}\rm
Much of this paper was inspired by ideas in \cite{flon2004ramification}, which gave a treatment, up to some imprecise language, of a problem that is very similar in spirit.  In this paper, the authors have decided to follow a different approach, similar to that of Oda in \cite{oda1995note}, which the authors believe to be more direct and more instructive to the reader. Namely, we reduce to the equicharacteristic case via a $2$-dimensional ring (see \S\ref{sec3.3}) and then relate the topological action to the $R=\mathbb{C}[[x]]$ case directly.

\end{rmk}

\subsection{Outline of the paper} \label{sec1.2}

In \S\ref{sec2}, we will use purely topological constructions and arguments to describe the action $\rho_{\tp}$ in terms of Dehn twists (Theorem \ref{thm main topological}).   In \S\ref{sec3}, we will construct a specific basis of $\pi_{1}(\mathcal{F}_{z_{0}}, \infty_{z_0})$ for which we can describe the action of $\rho_{\tp}$ explicitly.  Then we give some basic examples of how one can combine Theorem \ref{thm comparison} with the explicit formulas given in \S\ref{sec3.1} to give explicit generators and relations of the prime-to-$p$ \'etale fundamental groups of $\mathbb{P}^1_K$ minus $K$-rational branch points.  Our formulas will also be used to prove (Corollary \ref{prop field of def}) that the degree over $K$ of the field of moduli of any prime-to-$p$ $G$-Galois cover of $\proj_{\bar K}^1\smallsetminus\{\alpha_1,...,\alpha_d\}$ (together with an action of $G$) divides the exponent of $G$ modulo its center.  Finally, the whole of \S\ref{sec4} is dedicated to a proof of Theorem \ref{thm comparison}.

\subsection{Acknowledgments} \label{sec1.3}

The first author would like to thank David Harbater for introducing him to this topic and for many conversations related to this topic during his PhD studies at the University of Pennsylvania; Brian Conrad for several conversations on the topic during the author's time as a Szeg\"{o} assistant professor at Stanford; Andrew Obus for his help with a subtle issue in one of the arguments in the proof of Proposition \ref{prop isomorphic to family over cc((x))}; and Daniel Litt for an insightful correspondence following the release of the first version of the arXiv submission. The second author would like to thank Fabrizio Andreatta for many enlightening discussions and for providing suggestions and references relating to several crucial ideas. Finally, both authors are grateful to the referee for providing helpful suggestions which have improved the exposition in this paper.

\section{The topological monodromy action} \label{sec2}

We begin this section with a brief overview of some topological preliminaries which we will need below.  We will then construct some simple loops whose images lie in $(\proj_{\cc}^{1})^{\an} \smallsetminus \{a_{1}(z_{0}), ... , a_{d}(z_{0})\}$ and present and prove Theorem \ref{thm main topological}, which describes the monodromy action $\rho_{\tp}$ in terms of Dehn twists associated to these loops.

\subsection{Configuration spaces, mapping class groups, and Dehn twists} \label{sec2.1}

For any integer $d \geq 2$, we define $Y_{d}$ to be the complex manifold $\cc^{d} \smallsetminus \bigcup_{1 \leq i < j \leq d} \Delta_{i, j}$, where each $\Delta_{i, j}$ is the ``weak diagonal" subspace of $\cc^{d}$ consisting of points $(z_{1}, ... , z_{d})$ with $z_{i} = z_{j}$.  We endow $Y_{d} \subset \cc^{d}$ with the subspace topology and call it the \emph{ordered configuration space} of $d$-element subsets of $\cc$.  Each point $(z_{1}, ... , z_{d}) \in Y_{d}$ may be identified with the ordered $d$-element subset $\{z_{1}, ... , z_{d}\} \subset \cc$.

Given an integer $d \geq 2$ and any basepoint $T_{0} \in Y_{d}$, we write $\mathcal{Y}_{d}$ for the group of self-homeomorphisms of the complex plane which fix the subset $T_{0} \subset \cc$ pointwise (note that the structure of $\mathcal{Y}_{d}$ as an abstract topological group does not depend on our choice of $T_{0}$).  Additionally, let $\mathcal{Y}_{0}$ denote the group of all self-homeomorphisms of the complex plane.  We endow these groups of homeomorphisms with the compact-open topology and consider them as topological groups.  Note that we have an obvious inclusion $\iota : \mathcal{Y}_{d} \hookrightarrow \mathcal{Y}_{0}$.

For $d \geq 2$, the \emph{pure mapping class group (of the plane)} $\pi_{0}\mathcal{Y}_{d}$ is the quotient of the topological group $\mathcal{Y}_{d}$ modulo the path component of the identity $\id \in \mathcal{Y}_{d}$.  Note that each self-homeomorphism $f \in \mathcal{Y}_{d}$ induces an automorphism of $\pi_{1}((\proj_{\cc}^{1})^{\an} \smallsetminus T_{0}, \infty)$.  Since any self-homeomorphism in $\mathcal{Y}_{d}$ can be extended uniquely to a self-homeomorphism of $(\proj_{\cc}^{1})^{\an}$ which fixes the elements of $T_{0}$ as well as $\infty$, there is an obvious action of $\pi_{0}\mathcal{Y}_{d}$ on $\pi_{1}((\proj_{\cc}^{1})^{\an} \smallsetminus T_{0}, \infty)$, which we denote by $\varphi: \pi_{0}\mathcal{Y}_{d} \to \Aut(\pi_{1}((\proj_{\cc}^{1})^{\an} \smallsetminus T_{0}, \infty))$.

Let $\gamma : [0, 1] \to (\aff_{\cc}^{1})^{\an} \smallsetminus T_{0}$ be a simple loop, and write $\im(\gamma)$ for its image in $(\aff_{\cc}^{1})^{\an} \smallsetminus T_{0}$.  The simple loop $\im(\gamma)$ is homeomorphic to the unit circle $S^{1} := \{z \in \cc \ | \ |z| = 1\}$, and this homeomorphism can be chosen such that the image of $\gamma(t)$ is $e^{2 \pi \sqrt{-1}t} \in S^{1}$ for $t \in [0, 1]$.  Take a small tubular neighborhood around $\im(\gamma)$ on $(\aff_{\cc}^{1})^{\an} \smallsetminus T_{0}$, which is homeomorphic to $S^{1} \times [-\xi, \xi]$ for some small $\xi > 0$.  By abuse of notation, we identify this neighborhood with $S^{1} \times [-\xi, \xi]$; the outer edge is $S^{1} \times \{\xi\}$, and the inner edge is $S^{1} \times \{-\xi\}$.  We define the \emph{Dehn twist} $D_{\gamma} : (\aff_{\cc}^{1})^{\an} \to (\aff_{\cc}^{1})^{\an}$ to be the self-homeomorphism which acts as the identity on $(\aff_{\cc}^{1})^{\an} \smallsetminus S^{1} \times [-\xi, \xi]$ and which takes $(e^{2 \pi \sqrt{-1}t}, s) \in S^{1} \times [-\xi, \xi]$ to $(e^{2 \pi \sqrt{-1}(t + 1/2 - s/(2\xi))}, s)$.  We may visualize $D_{\gamma}$ as a self-homeomorphism of $(\aff_{\cc}^{1})^{\an}$ that keeps the outer edge of the tubular neighborhood fixed while twisting the inner edge one full rotation counterclockwise.  Clearly, $D_{\gamma} \in \mathcal{Y}_{d}$; we also denote its path component in $\pi_{0}\mathcal{Y}_{d}$ by $D_{\gamma}$.

Let $\epsilon : \mathcal{Y}_{0} \to Y_{d}$ be the \emph{evaluation map} defined by sending any homeomorphism $f \in \mathcal{Y}_{0}$ to $(f(z_{1}), ... , f(z_{d})) \in Y_{d}$.  Note that the inverse image of $T_{0} = (z_{1}, ... , z_{d}) \in Y_{d}$ under $\epsilon$ is $\mathcal{Y}_{d}$.  It follows easily from \cite[Theorem 4.1]{birman1974braids} that this map is a fiber bundle.  Therefore, $\epsilon$ induces a long exact sequence of fundamental groups 
\begin{equation}\label{homotopy long exact sequence}
 ... \stackrel{\iota_{*}}{\to} \pi_{1}(\mathcal{Y}_{0}, \id) \stackrel{\epsilon_{*}}{\to} \pi_{1}(Y_{d}, T_{0}) \stackrel{\partial}{\to} \pi_{0}\mathcal{Y}_{d} \stackrel{\iota_{*}}{\to} \pi_{0}\mathcal{Y}_{0} \stackrel{\epsilon_{*}}{\to} \pi_{0}Y_{d} = 1.
\end{equation}
We observe that by \cite[Theorem 4.4]{birman1974braids}, the mapping class group $\pi_{0}\mathcal{Y}_{0}$ is trivial, and therefore the map $\partial : \pi_{1}(Y_{d}, T_{0}) \to \pi_{0}\mathcal{Y}_{d}$ is surjective. This map can be described explicitly as follows.  Let $\gamma : [0, 1] \to Y_{d}$ be a loop, with $\gamma(0) = \gamma(1) = T_{0}$, and let $[\gamma] \in \pi_{1}(Y_{d}, T_{0})$ be the corresponding equivalence class.  Then $\gamma$ lifts to a path $\tilde{\gamma} : [0, 1] \to \mathcal{Y}_{0}$ with $\tilde{\gamma}(0) = \id$, via the fiber bundle $\epsilon : \mathcal{Y}_{0} \to Y_{d}$.  Note that the self-homeomorphism $\tilde{\gamma}(1) : (\aff_{\cc}^{1})^{\an} \to (\aff_{\cc}^{1})^{\an}$ fixes the points $z_{1}, ... , z_{d}$ since $\gamma(1) = (z_{1}, ... , z_{d})$, so $\tilde{\gamma}(1) \in \mathcal{Y}_{d}$.  Then $\partial([\gamma]) \in \pi_{0}\mathcal{Y}_{d}$ is the path component of $\tilde{\gamma}(1)$.

Note that we have a surjective map $Y_{d + 1} \to Y_{d}$ defined by ``forgetting" the $(d + 1)^{\oper{st}}$ point, under which the inverse image of any point $T \in Y_{d}$ is $(\aff_{\cc}^{1})^{\an} \smallsetminus T$.  According to \cite[Theorem 3]{fadell1962configuration}, this map is a fiber bundle.  It is also shown in \cite{fadell1962configuration} that $Y_{d}$ (and $Y_{d + 1}$) have trivial $j^{\oper{th}}$ homotopy groups for $j \geq 2$.

Now define the family $\bar{Y}_{d + 1} \to Y_{d}$ to be the family $Y_{d + 1} \to Y_{d}$ ``with the infinity section added"; that is, $\bar{Y}_{d + 1} \to Y_{d}$ is the subfamily of the isotrivial family $(\proj_{\cc}^{1})^{\an} \times Y_{d} \to Y_{d}$ such that over each $T \in Y_{d}$, the fiber $(\bar{Y}_{d + 1})_{T}$ is $(\proj_{\cc}^{1})^{\an} \smallsetminus T$.  For each $T \in Y_{d}$, write $\infty_{T}$ for the point in $\bar{Y}_{d + 1}$ coming from the point at infinity in the fiber $(\bar{Y}_{d + 1})_{T} = (\proj_{\cc}^{1})^{\an} \smallsetminus T$.  It is clear that $\bar{Y}_{d + 1} \to Y_{d}$ is also a fiber bundle.  The associated long exact homotopy sequences for the fiber bundles $Y_{d + 1} \to Y_d$ and $\bar{Y}_{d + 1} \to Y_{d}$ both truncate. Moreover, the fiber bundle $Y_{d + 1} \to Y_d$ admits a section $s : Y_{d} \to Y_{d + 1}$ whose $(d + 1)$th coordinate is given by $T \mapsto 1 + \max_{z \in T} |z|$. Considered as a map $Y_d \to \bar{Y}_{d + 1}$, $s$ deforms to the ``infinity section" $\underline{\infty}$ whose $(d + 1)$th coordinate is given by $T \mapsto \infty_{T}$ of the fiber bundle $\bar{Y}_{d + 1} \to Y_d$.  We therefore have the following commutative diagram, where the top and bottom rows are the short exact sequences associated to these two fiber bundles, split via the sections $s$ and $\underline{\infty}$.

\begin{equation} \label{eq S4} \xymatrix{ 1 \ar[r] & \pi_1((\mathbb{A}^1_{\mathbb{C}})^{\an} \smallsetminus T_0, s(T_0)) \ar[r] \ar@{->>}[d] & \pi_1(Y_{d+1}, s(T_0)) \ar[r] \ar@{->>}[d] & \pi_1(Y_{d}, T_0) \ar@/_1pc/[l]_{s_*} \ar[r] \ar@{=}[d] & 1 
\\ 1 \ar[r] &  \pi_1((\mathbb{P}^1_{\mathbb{C}})^{\an} \smallsetminus T_0, \infty) \ar[r] & \pi_1(\bar{Y}_{d + 1},\underline{\infty}(T_0)) \ar[r] & \pi_1(Y_{d}, T_0) \ar@/_1pc/[l] \ar@/_1pc/[l]_{\underline{\infty}_*} \ar[r] & 1} \end{equation}

The splitting of the short exact sequence on the bottom row of (\ref{eq S4}) gives rise to a monodromy action of $\pi_{1}(Y_{d}, T_{0})$ on $\pi_{1}((\proj_{\cc}^{1})^{\an} \smallsetminus T_{0}, \infty)$, which we denote by $\rho : \pi_{1}(Y_{d}, T_{0}) \to \Aut(\pi_{1}((\proj_{\cc}^{1})^{\an} \smallsetminus T_{0}, \infty))$.



\begin{prop} \label{prop factors through MCG}

For any basepoint $T_0 \in Y_{d}$, the representation $\rho : \pi_{1}(Y_{d}, T_0) \to \Aut(\pi_{1}((\proj_{\cc}^{1})^{\an} \smallsetminus T_0, \infty))$ factors through $\pi_{0}\mathcal{Y}_{d}$ as the composition $\varphi \circ \partial$.

\end{prop}

\begin{proof}

To simplify notation, we will assume that $T_0 = (1, ... , d)$. Let $\beta : [0, 1] \to Y_{d}$ be a loop based at $(1, ... , d)$, which, via the infinity section, induces a loop on $\bar{Y}_{d + 1}$ which we also denote by $\beta$. Let $\gamma : [0, 1] \to (\proj_{\cc}^{1})^{\an} \smallsetminus \{1, ... , d\} \hookrightarrow \bar{Y}_{d + 1}$ be a loop based at $\infty$.  We want to show that the concatenation of loops $\beta^{-1}\gamma\beta$ is a representative of the element $\varphi(\partial(\beta))([\gamma]) \in \pi_{1}((\proj_{\cc}^{1})^{\an} \smallsetminus \{1, ... , d\}, \infty)$.  It is clear from the above discussion and the diagram in (\ref{eq S4}) that we may deform $\beta$ and $\gamma$ so that $\beta$ is a loop $[0, 1] \to Y_{d} \stackrel{s}{\hookrightarrow} \bar{Y}_{d + 1}$ and $\gamma$ is a loop on $(\aff_{\cc}^{1})^{\an} \smallsetminus \{1, ... , d\}$ based at $d + 1$. Then it suffices to show that the concatenation of loops given by $\beta^{-1}\gamma\beta$ is homotopic to the loop obtained by acting on $\gamma$ by a representative of $\partial([\beta])$ in $\mathcal{Y}_{d}$.

From the discussion above, we see that the loops $\beta, \gamma : [0, 1] \to Y_{d + 1}$ lift via the fiber bundle $\epsilon : \mathcal{Y}_{0} \to Y_{d + 1}$ to paths $\tilde{\beta}, \tilde{\gamma} : [0, 1] \to \mathcal{Y}_{0}$ starting at $\id \in \mathcal{Y}_{0}$ and ending at representatives of $\partial([\beta]), \partial([\gamma]) \in \pi_{0}\mathcal{Y}_{d}$ respectively.  Therefore, $\beta^{-1}\gamma\beta$ lifts to the path $[0, 1] \to \mathcal{Y}_{0}$ given by $t \mapsto \tilde{\beta}(t)^{-1}\tilde{\gamma}(t)\tilde{\beta}(t)$.  We claim that this path is homotopic to the path $\tilde{\delta} : t \mapsto \tilde{\beta}(1)^{-1}\tilde{\gamma}(t)\tilde{\beta}(1)$.  Indeed, there is a homotopy $[0, 1] \times [0, 1] \to \mathcal{Y}_{0}$ of paths starting at $\id$ and ending at $\tilde{\beta}(1)^{-1}\tilde{\gamma}(1)\tilde{\beta}(1)$ given by $(s, t) \mapsto \tilde{\beta}(t + s(1 - t))^{-1}\tilde{\gamma}(t)\tilde{\beta}(t + s(1 - t))$ which deforms $\tilde{\beta}^{-1}\tilde{\gamma}\tilde{\beta}$ to $\tilde{\delta}$.  Now since $\tilde{\beta}(1)$ fixes each of the points $1, ... , d + 1$ while $\tilde{\gamma}(t)$ fixes each of the points $1, ... , d$ for all $t$, it is easy to check that $\tilde{\delta}(t)$ fixes each point $1, ... , d$ and takes $d + 1$ to $(\tilde{\beta}(1) \circ \gamma)(t)$ for all $t$. It follows that $\tilde{\delta} : [0, 1] \to \mathcal{Y}_{0}$ is the (unique) lifting of the loop obtained by acting on $\gamma$ by $\tilde{\beta}(1) \in \mathcal{Y}_{d}$, which is a representative of $\partial([\beta])$.  This implies the desired homotopy of loops on $\bar{Y}_{d + 1}$.

\end{proof}

\subsection{Description of the topological monodromy action} \label{sec2.2}

In Theorem \ref{thm main topological}  we will describe the topological monodromy action in terms of loops on $(\proj_{\cc}^{1})^{\an} \smallsetminus \{a_{1}(z_{0}), ... , a_{d}(z_{0})\}$ that depend on the intersection behavior of the $a_{i}$'s. These loops can be taken to be circles for a suitable choice of $z_{0}$. Since knowledge of the monodromy action does not depend on our choice of $z_0$, we will allow ourselves to restrict our attention to such a strategic choice.

We begin by introducing some notation. Let $F_n:\mathbb{C}[[x]] \rightarrow \mathbb{C}[x]$ be the map given by $\sum_{i = 0}^{\infty} c_ix^i \mapsto \sum_{i=0}^{n-1} c_ix^i$. Let $\mathcal{I}$ be the set of all pairs $(I, n)$ where $I \subseteq \{1, ... , d\}$ is a subset of cardinality $d \geq 2$ and $n \geq 1$ is an integer such that $e_{i, j} \geq n$ for all $i, j \in I$ and such that $I$ is maximal among subsets with this property. We remark that if  $(I,n)$ is in $\mathcal{I}$, then for all $i,j\in I$ we have that $F_n(a_i)=F_n(a_j)$. For such a pair $(I,n)$, we denote this polynomial by $b_{I,n}$, and let $w_{I, n} = b_{I,n}(z_{0})$.

Let $\eta, r > 0$, and let $\gamma_{I, n}$ be the loop given by $t \mapsto w_{I, n} + r^{n - 1} \eta e^{2\pi \sqrt{-1} t}$.  We write $\im(\gamma_{I, n})$ for its image in $(\aff_{\cc}^{1})^{\an}$, and write $B_{I, n} := \{z \in (\aff_{\cc}^{1})^{\an} \ | \ |z - w_{I, n}| < r^{n - 1} \eta\}$ for the simply connected component of $(\aff_{\cc}^{1})^{\an} \smallsetminus \im(\gamma_{I, n})$. Note that $\gamma_{I,n}$ depends also on $\eta$ and $r$, but we suppress this in the notation.

\begin{prop} \label{prop loops around clusters of points}

Assume all of the above notation.  For every $\eta > 0$ sufficiently small, there exists an $r>0$ sufficiently small, so that if 
$\frac{r}{2}<|z_{0}|<r$ then the following hold:

a) For distinct pairs $(I, n), (I', n') \in \mathcal{I}$, the images $\im(\gamma_{I, n})$ and $\im(\gamma_{I', n'})$ do not intersect; and 

b) For any pair $(I, n) \in \mathcal{I}$, we have $a_{i}(z_{0}) \in B_{I, n}$ for $i \in I$ and $a_{j}(z_{0}) \notin B_{I, n} \cup \im(\gamma_{I, n})$ for $j \notin I$.

\end{prop}

In order to prove this proposition, we require the following lemma.

\begin{lemma} \label{prop clusters of points}

For every $\eta > 0$, there exists an $r > 0$ sufficiently small such that for every $(I,n)\in \mathcal{I}$ and for every $i \in I$ and $z\in B_{r}$ we have 
$$|a_{i}(z) - b_{I, n}(z)| < |z|^{n - 1} \eta.$$

\end{lemma}

\begin{proof}

Since $\mathcal{I}$ is finite, it suffices to prove the statement for a particular $(I, n)$. Let $a_{i,n}=x^{-n}(a_{i} 
- b_{I, n})$, and note that it converges wherever $a_{i}$ does.  The continuity of the map $\cc \to \cc^{\#I}$ that takes $z \mapsto (z\cdot a_{i,n}(z))_{i \in I}$ implies that given $\eta > 0$, there is a $r> 0$ such that $|z\cdot a_{i,n}(z)| < \eta$ for $i \in I$ and any $z \in B_{r}$.  Thus, we have $|a_{i}(z) - b_{I, n}(z)| = |z^{n}a_{i,n}(z)| = |z|^{n - 1}|z\cdot a_{i,n}(z)| < |z|^{n - 1} \eta$ for $i \in I$.

\end{proof}

We are now ready to prove Proposition \ref{prop loops around clusters of points}.

\begin{proof} (of Proposition \ref{prop loops around clusters of points})

Choose distinct pairs $(I, n), (I', n') \in \mathcal{I}$, and assume without loss of generality that $I$ is not strictly contained in $I'$.  It then follows from obvious properties of the $x$-adic valuation that either $I \supseteq I'$ or $I \cap I' = \varnothing$.

First assume that $I \supseteq I'$. Then we may assume without loss of generality that $n < n'$ (which is automatically the case if $I\supsetneq I'$).  Let $z$ satisfy $|z-w_{I',n'}| = \eta r^{n'-1}$. Then we have 
$$|z-w_{I,n}| \leq |z-w_{I',n'}|+|w_{I,n}-w_{I',n'}| = \eta r^{n'-1} + |z_0|^n |z_0^{-n}(w_{I,n}-w_{I',n'})|$$
\begin{equation} <\eta r^{n'-1}+r^n|z_0^{-n}(w_{I,n}-w_{I',n'})|\leq (\eta+|z_0^{-n}(w_{I,n}-w_{I',n'})|)r^n, \end{equation}
where the last inequality holds if $r$ is chosen to be less than $1$. Since $x^n$ divides $b_{I,n}-b_{I',n'}$, it follows that we may choose a sufficiently small $r$, depending on $\eta$, but independent of the choice of $z_0$, so that $(\eta + |z_0^{-n}(w_{I,n}-w_{I',n'})|)r<\eta$. Thus $|z-w_{I,n}|<\eta r^{n-1}$.

Now assume that $I\cap I'=\varnothing$. Given what we have proven above, we may assume without loss of generality that $n = n'$ and that $x^{n-1}$ divides $b_{I,n}-b_{I',n}$. Since $b_{I,n}$ and $b_{I',n}$ both have degree at most $n-1$, this implies that $b_{I,n}-b_{I',n} = c x^{n-1}$ for some nonzero constant $c$.  For any choice of $\eta < 2^{-n}|c|$, we therefore have $|w_{I,n}-w_{I',n}| = |c| |z_0|^{n - 1} > \eta|z_0|^{n-1}2^{n}$. Since we have assumed that $\frac{r}{2}<|z_0|$, it follows that $|w_{I,n}-w_{I',n}|>\eta(\frac{r}{2})^{n-1}2^{n}=2r^{n-1}\eta$. This proves part (a).

It follows immediately from Lemma \ref{prop clusters of points} that we may choose $r$ small enough so that for every $(I, n) \in \mathcal{I}$ and $i\in I$, the point $a_{i}(z_{0})$ will be in $B_{I,n}$. Now choose $j \notin I$; to prove part (b), it suffices to show that $a_{j}(z_{0}) \notin B_{I, n} \cup \im(\gamma_{I, n})$. The maximality of $I$ implies that $F_n(a_{j}) \neq b_{I, n}$, and therefore, $x_{n}$ does not divide $a_j-b_{I,n}$.  It follows that there is a constant $c' > 0$ such that $|a_{j}(z) - b_{I, n}(z)| \geq c'|z|^{n - 1}$ for $z \in B_{r}^{*}$ if $r$ is chosen to be small enough,  independently of $\eta$.  Assume that we have chosen $\eta$ so that $\eta < 2^{-(n - 1)}c'$.  Then we have 
\begin{equation} |a_{j}(z_{0}) - w_{I, n}| = |a_{j}(z_{0}) - b_{I, n}(z_{0})| \geq |z_{0}|^{n - 1}  c' \geq 2^{-(n - 1)}r^{n - 1} c' > \eta r^{n - 1}. \end{equation}
It follows that $a_{j}(z_{0}) \notin B_{I, n} \cup \im(\gamma_{I, n})$, and part (b) is proved.

\end{proof}

Each $\gamma_{I,n}$ defined above induces a induces a Dehn twist which we denote by $D_{I, n} \in \pi_{0}\mathcal{Y}_{n}$. The proposition above implies that for every $\eta>0$ small enough there exists an $r>0$ small enough for which the the $D_{I, n}$'s commute.

For ease of notation, given any $t \in \rr$, we write $e(t)$ to mean $e^{2\pi \sqrt{-1}t}$.  Let $\delta \in \pi_{1}(B_{\varepsilon}^{*}, z_{0})$ be the homotopy equivalence class of the loop given by $t \mapsto e(t)z_{0}$ for $t \in [0, 1]$.  Clearly, $\delta$ is a generator of $\pi_{1}(B_{\varepsilon}^{*}, z_{0}) \cong \zz$.  Therefore, in order to determine the monodromy action  $\rho_{\tp}$, it suffices to know how $\delta$ acts on $\pi_{1}(\mathcal{F}_{z_{0}}, \infty)$.

We are finally ready to state our main topological result.

\begin{thm} \label{thm main topological}

For every $\eta > 0$ sufficiently small, there exists an $r > 0$ sufficiently small so that if $\frac{r}{2} < |z_0| < r$, then the generator $\delta \in \pi_{1}(B_{\varepsilon}^{*}, z_{0})$ acts via $\rho_{\tp}$ on $\pi_{1}(\mathcal{F}_{z_{0}}, \infty_{z_0})$ in the same way that the product of Dehn twists $\prod_{(I, d) \in \mathcal{I}} D_{I, n}$ does; in other words, $\rho_{\tp}(\delta) = \varphi(\prod_{(I, d) \in \mathcal{I}} D_{I, n})$.

\end{thm}

\begin{rmk}\rm  \,\, {\, \, \mbox{ }} \label{rmk intersection behavior}

(a) The above theorem implies in particular that in this situation, the monodromy action depends only on the intersection behavior of the power series over the disk $B_{\varepsilon}$.

(b) The above theorem can be viewed as a generalization of a result of Oda given by \cite[Main Lemma 1.7]{oda1995note}. Oda's lemma states that under certain technical hypotheses regarding their degenerations, a generator of the monodromy action associated to a family of \it compact \rm genus-$g$  Riemann surfaces over the punctured disk acts as a certain product of Dehn twists. Indeed, for every $g\geq 0$ and compact genus-$g$ Riemann surface there exists a Zariski open subset that can be viewed as degree-$2$ covering space of the sphere minus $2g + 2$ points.

\end{rmk}

\subsection{Proof of Theorem \ref{thm main topological}} \label{sec2.3}

In order to prove Theorem \ref{thm main topological}, we first need several lemmata.  To simplify notation in the lemma below, we make the following extra assumption regarding the power series $a_{i}$, which always holds after a suitable reordering.

\begin{hyp} \label{hypothesis decreasing intersection numbers}

For every $m \in \{1, ... , d\}$, the function $\{m + 1, ... , d\} \to \zz_{\geq 0}$ given by $i \mapsto e_{m, i}$ is (weakly) monotonically decreasing.

\end{hyp}

Note that this assumption implies that the subsets $I \subseteq \{1, ... , d\}$ such that $(I, n) \in \mathcal{I}$ for some $n \geq 1$ are always subintervals $\{m, ... , m + l - 1\}$ with $1 \leq m \leq d$ and $l \geq 2$.

For $I=  \{m, ... , m + l - 1\}$ with $(I,n)\in\mathcal{I}$, let $\lambda_{I,n}$ be the loop in $Y_d$ given by: $$(((a_{i}(z_{0}))_{1 \leq i \leq m - 1}, (w_{I,n} + e(t) (a_i(z_0)-w_{I,n}))_{m \leq i \leq m + l - 1}, ((a_{i}(z_{0}))_{m + l \leq i \leq d}))_{0\leq t\leq 1})$$

\begin{lemma} \label{lemma key topological}

Let $J = \{m, ... , m + l - 1\}$ be a subinterval such that $(J, 1) \in \mathcal{I}$. For each $i \in J$, let $c_{i}(x) = x^{-1} (a_{i}(x) - w) \in \cc[[x]]$, where $w= a_{m}(0) = ... = a_{m + l - 1}(0)$. Let $r' > 0$ be small enough so that for all $z \in  B_{r'}$ and $i = 1, ..., d$ we have $|a_{i}(z) - a_{i}(0)| < \mu := \frac{1}{2}\min_{i \notin J} |a_{i}(0) - w|$. Then the loop in $Y_d$ given by 
$$(a_{1}(e(t) z_{0}), ... , a_{d}(e(t) z_{0}))_{0\leq t\leq 1}$$
is homotopic to the concatenation of $\lambda_{J,1}$ with the loop in $Y_d$ given by
$$((a_{i}(e(t)z_{0}))_{1 \leq i \leq m - 1}, (w + z_{0}c_{i}(e(t)z_{0}))_{m \leq i \leq m + l - 1}, (a_{i}(e(t)z_{0}))_{m + l \leq i \leq d})_{0\leq t\leq 1}.$$
 
\end{lemma}

\begin{proof}

We define the continuous map $H : [0, 1] \times [0, 1] \to \cc^{d}$ as follows.

$$H(s, t) = \begin{cases}
\Big((a_{i}(z_{0}))_{1 \leq i \leq m - 1}, \\
\hspace{2em} (w + e(\tfrac{t}{1 - s/2})z_{0}c_{i}(z_{0}))_{m \leq i \leq m + l - 1}, ((a_{i}(z_{0}))_{m + l \leq i \leq d}\Big) & 0 \leq t \leq s/2 \\
\Big((a_{i}(e(\tfrac{t - s/2}{1 - s/2})z_{0}))_{1 \leq i \leq m - 1}, \\
\hspace{2em} (w + e(\tfrac{t}{1 - s/2})z_{0}c_{i}(e(\tfrac{t - s/2}{1 - s/2})z_{0}))_{m \leq i \leq m + l - 1}, (a_{i}(e(\tfrac{t - s/2}{1 - s/2})z_{0}))_{m + l \leq i \leq d}\Big) & s/2 \leq t \leq 1 - s/2 \\
\Big((a_{i}(e(\tfrac{t - s/2}{1 - s/2})z_{0}))_{1 \leq i \leq m - 1}, \\
\hspace{2em} (w + z_{0}c_{i}(e(\tfrac{t - s/2}{1 - s/2})z_{0}))_{m \leq i \leq m + l - 1}, (a_{i}(e(\tfrac{t - s/2}{1 - s/2})z_{0}))_{m + l \leq i \leq d}\Big) & 1 - s/2 \leq t \leq 1
\end{cases}$$

It is straightforward to check from the formulas given above that $H$ is continuous; that $H(0, t)$ is the original loop; and 
that $H(1, t)$ is a concatenation of the loops described in the statement.  It therefore suffices to show that $H(s, t) \in 
Y_{d}$ for all $(s, t) \in [0, 1] \times [0, 1]$. 

It is easy to verify from Hypothesis \ref{hypothesis varepsilon} that the $c_{i}$'s take distinct values on $B_{r'}^{*}$ for $i \in J$.  It follows that the $i$th and $j$th coordinates of $H(s, t)$ are different for all $(s, t)$ if $i, j \in J$ or if $i, j \notin J$.  Now it clearly suffices to show that the $i$th coordinate of $H(s, t)$, which we denote by $H(s, t)_{i}$, satisfies $|H(s, t)_{i} - w| < \mu$ if and only if $i \in J$.  Choose some $j \in J$ and some $i \notin J$, and let $I$ be the unique subinterval containing $i$ such that $(I, 1) \in \mathcal{I}$.  Now it can be directly verified from the formulas that $|H(s, t)_{j} - w| = |a_{j}(e(u)z_{0}) - w| < \mu$, and that meanwhile, $|H(s, t)_{i} - a_{i}(0)| = |a_{i}(e(u)z_{0}) - a_{i}(0)| < \mu$, where $u = 1$ if $0 \leq t \leq s/2$ and $u = \frac{t - s/2}{1 - s/2}$ otherwise.  Therefore, we have $|H(s, t)_{j} - w| < \mu$, and since $|a_{i}(0) - w| \geq 2\mu$, we have $|H(s,t)_i(0) - w| \geq \mu$.  Since $j \in J$ and $i \notin J$ were chosen arbitrarily, we get the desired statement.
  
\end{proof}

\begin{lemma} \label{lemma braids mapping to Dehn twists}

For choices of $\eta$, $r$ and $z_0$ as in Proposition \ref{prop loops around clusters of points}, the loop $\lambda_{I,n}$ defined above acts on 
$\pi_{1}((\aff_{\cc}^{1})^{\an} \smallsetminus T_{0}, \infty)$ as $\varphi(D_{I, n})$.
\end{lemma}

\begin{proof}

By Proposition \ref{prop factors through MCG}, it is enough to show that $\partial([\lambda_{I, n}]) = D_{I, n} \in \pi_{0}\mathcal{Y}_{d}$.  Choose some real number $\xi > 0$ small enough so that $\xi < r^{n - 1}\eta$ and $\{z \in \cc \ | \ r^{n - 1}\eta - \xi < |z - w| < r^{n - 1}\eta + \xi\} \subset (\aff_{\cc}^{1})^{\an}$ does not contain any of the points $a_{i}(z_{0})$.  Let $\tilde{\lambda}_{I, n} : [0, 1] \times \mathcal{Y}_{0} \to (\aff_{\cc}^{1})^{\an} \smallsetminus T_{0}$ be the homotopy such that for all $t \in [0, 1]$, $\tilde{\lambda}_{I, n}(t) : (\aff_{\cc}^{1})^{\an} \to (\aff_{\cc}^{1})^{\an}$ acts on $\{z \in \cc \ | \ |z - w_{I, n}| > r^{n - 1}\eta + \xi\}$ as the identity, on $\{z \in \cc \ | \ |z - w_{I, n}| < r^{n - 1}\eta - \xi\}$ as $z \mapsto e(t)z$, and on $\{z \in \cc \ | \ r^{n - 1}\eta - \xi < |z - w_{I, n}| < r^{n - 1}\eta + \xi\}$ by fixing the outer rim and twisting the inner rim counterclockwise by an angle of $2\pi t$.  Then clearly $\tilde{\lambda}_{I, n}(0)$ is the identity $\id \in \mathcal{Y}_{0}$, and $\tilde{\lambda}_{I, n}(t)$ agrees with $\lambda_{I, n}(t)$ on the points $a_{1}(z_{0}), ... , a_{d}(z_{0})$ for all $t \in [0, 1]$.  It follows from the construction of $\partial$ given above that $\partial([\lambda_{I, n}])$ is represented by $\tilde{\lambda}_{I, n}(1)$ in $\pi_{0}\mathcal{Y}_{d}$.  Since by definition, the Dehn twist $D_{I, n} \in \pi_{0}\mathcal{Y}_{d}$ is also represented by $\tilde{\lambda}_{I, n}(1)$, we are done.

\end{proof}

\begin{lemma} \label{ugly}

Let $p_1,...,p_d$ be power series which converge on $B_{\nu}$, for some $\nu>0$.  Define $\mathcal{I}$ for these power series as above, and assume that $(J,1)$ is in $\mathcal{I}$. Choose $z_0 \in B_{\nu}^*$, and let $p'_i = p_i$ for $i\not \in J$ and $p'_i=p_i(0)+z_0x^{-1}(p_i-p_i(0))$ for $i\in J$. Then there exists a $\nu' > 0$ which is independent of $z_0$, such that for all $z \in B_{\nu'}^*$ and $i = 1,...,d$ we have $|p'_{i}(z) - p'_{i}(0)| <  \frac{1}{2} \min_{p'_j(0) \neq p'_i(0)} |p'_{i}(0) - p'_j(0)|$.

\end{lemma}

\begin{proof}

Fix the notation $p_i=\sum_{k=1}^{\infty}b_{i,k}x^k$ for $i=1,...,d$. Fix an index $i$, and let $j$ be such that $p'_j(0)\neq p'_i(0)$.

If $i, j \notin J$, then $p'_i = p_i$ and $p'_j = p_j$ are defined independently of $z_0$. Since $p_i(z)-p_i(0)$ has no constant term, it follows that $\nu'$ can be chosen (independently of $z_0$) to be small enough that for $z \in B_{\nu'}^*$, we have
\begin{equation} |p_i(z)-p_i(0)| < \tfrac{1}{2} |p_i(0) - p_j(0)|\neq 0. \end{equation}
\textit
If $i \not \in J$ and $j \in J$, then $|p'_i(0)-p'_j(0)|=|p_i(0)-p_j(0)-z_0b_{j,1}|$. Note that one can choose $\nu'$ to be small enough that $|p_i(0)-p_j(0)-z_0b_{j,1}|$ is arbitrarily close to $|p_i(0)-p_j(0)|\neq 0$. Since $p'_{i}(z) - p'_{i}(0)$ has no constant term, it follows that $\nu'$ can be chosen (independently of $z_0$) to be small enough so that for all $z \in B_{\nu'}^{*}$, we have 
\begin{equation} |p'_{i}(z) - p'_{i}(0)| < \tfrac12 |p'_i(0)-p'_j(0)|. \end{equation}

If $i \in J$ and $j\not \in J$, then
$|p'_i(0) - p'_j(0)| = |p_i(0) + z_0b_{i1} - p_j(0)|$, and $\nu'$ can be chosen (independently of $z_0$) so that $|p'_i(0)-p'_j(0)|$ is arbitrarily close to $|p_i(0) - p_j(0)| \neq 0$. The proof then follows as in the previous case.

If $i, j\in J$, then $p_i(0) = p_j(0)$, and therefore $|p'_i(0) - p'_j(0)| = |p_i(0) + z_0b_{i,1} - p_j(0) - z_0b_{j,1}| = |z_0||b_{i,1} - b_{j,1}|$. Meanwhile, $|p'_{i}(z) - p'_{i}(0)|=|z_0||(p_i - b_{i,0} -b_{i,1}z)/z|$. Therefore, the inequality $|p'_{i}(z) - p'_{i}(0)| < 
\tfrac12 |p'_{i}(0) - p'_j(0)|$ simplifies to 
\begin{equation} \label{eq ugly} |(p_i - b_{i,0} - b_{i,1}z)/z| < \tfrac12 |b_{i,1}-b_{j,1}| \neq 0. \end{equation}
Since $(a_i - a_{i,0} - a_{i,1}z)/z$ is a power series with no constant term, again $\nu'$ can be chosen (independently of $z_0$) so that 
(\ref{eq ugly}) holds for all $z \in B_{\nu'}^{*}$.

\end{proof}

We are now ready to finish the proof of the main theorem of this section.

\begin{proof} (of Theorem \ref{thm main topological})

If $\mathcal{I}$ is empty, then $a_{i}(0) \neq a_{j}(0)$ for $i \neq j$.  In this case, $\mathcal{F} \to B_{\varepsilon}$ is clearly a trivial fiber bundle and the action of $\delta$ is trivial.  We therefore assume that $\mathcal{I}$ is not empty, so $N := \#\mathcal{I} \geq 1$.

Let $\eta$, $r$ and $z_0$ be as in Proposition \ref{prop loops around clusters of points}. Choose an interval $J = \{m, ... , m + l - 1\} \subseteq \{1, ... , d\}$ such that $(J, 1) \in \mathcal{I}$.  Now define power series $a_{i}' \in \cc[[x]]$ by setting $a_{i}' = a_{i}$ for $i \in \{1, ... , m - 1, m + l, ... , d\}$ and $a_{i}' = w + z_{0}c_{i}$, where $w$ and $c_{i}$ are defined as in the statement of Lemma \ref{lemma key topological}.  Note that $a_{i}'(z_{0}) = a_{i}(z_{0})$ for $1 \leq i \leq d$.   Define $\mathcal{I}'$ for the power series $a_{i}'$ in the same way that $\mathcal{I}$ was defined for the power series $a_{i}$; note that $\mathcal{I}'$ has cardinality $N - 1$. For any $(I', n') \in \mathcal{I}$, let $\lambda'_{I', n'} : [0, 1] \to Y_{d}$ be the path defined with respect to the $a'_i$'s analogously to how the $\lambda_{I, n}$'s were defined with respect to the $a_i$'s. Note that for any subinterval $J' \subseteq J$, $(J',1) \in \mathcal{I'}$ if and only if $(J',2) \in \mathcal{I}$ and that we have $\lambda'_{J',1} = \lambda_{J', 2}$ in this case. The rest of the argument follows from an argument by induction on $N$. That is, we assume that the statement holds if $\#\mathcal{I} = N - 1$ and therefore that it holds for the power series $a_{i}'$. Then after choosing $r$ to be small enough, we see by using Lemmas \ref{lemma key topological} and \ref{lemma braids mapping to Dehn twists} that the generator $\delta \in \pi_1(B_{\varepsilon}^*, z_0)$ acts as $\varphi(\prod_{(I', n') \in \mathcal{I}'} D_{I', n'})$ composed with $\varphi(D_{J, 1})$. The fact that the choice of $r$ at each step is independent of $z_0$ follows immediately from Lemma \ref{ugly}.

\end{proof}

\section{Applications and examples} \label{sec3}

In this section we provide an algorithm for explicit computations of the monodromy actions that are the subject of Theorem \ref{thm comparison}. This is done by first giving an algorithm for explicit computations for the topological monodromy and then using  Theorem \ref{thm comparison} to give explicit description of the arithmetic monodromy action. In \S\ref{sec3.3}, we will use this description to prove a result on fields of moduli (Proposition \ref{prop field of def}) which we will need in \S\ref{sec4} but which is interesting in its own right.

\subsection{Explicit algorithm for computing the topological monodromy action} \label{sec3.1}

Let $a_1,...,a_d$ be power series in $\mathbb{C}[[x]]$ satisfying Hypothesis \ref{hypothesis varepsilon}. In this section we will introduce a set of generators of $\pi_1(\mathcal{F}_{z_{0}}, \infty_{z_0})$ for which we give an explicit description of the monodromy action by $\pi_1(B^*_{\varepsilon},z_0)$. Our approach is similar to the one taken in \cite[\S6]{yu351toward}.

We reprise all of the notation used in \S\ref{sec2.1} and fix $T_{0} = (a_{1}(z_{0}), ... , a_{d}(z_{0})) \in Y_{d}$. Given an integer $d \geq 2$, the \textit{full braid group on $d$ strands}, denoted $B_{d}$, is defined to be the group of homotopy classes of paths in the ordered configuration space $Y_{d}$ of the affine line which begin at the point $(1, ... , d) \in Y_{d}$ and end at a point in $Y_{d}$ given by a permutation of the ordered set $(1, ... , d)$. The \textit{pure braid group on $d$ strands}, denoted $P_{d}$, is defined to be the fundamental group $\pi_1(Y_d,(1,...,d))$; we view it as a normal subgroup of $B_{d}$ in the obvious way.

There is a well-known description of $B_{d}$ as an abstract group, given by generators $\beta_{1}, ... , \beta_{d - 1}$ (where each $\beta_{i}$ is the braid rotating the $i^{\oper{th}}$ and $(i+1)^{\oper{st}}$ points in a counterclockwise semicircular motion) and relations 
\begin{equation} \label{braid relations}
\begin{cases}
\beta_{i}\beta_{j} = \beta_{j}\beta_{i},& |i - j| \geq 2 \\
\beta_{i}\beta_{i + 1}\beta_{i} = \beta_{i + 1}\beta_{i}\beta_{i + 1}, & 1 \leq i \leq d - 2.
\end{cases}
\end{equation}

We recall the fiber bundle $Y_{d + 1} \to Y_{d}$ defined in \S\ref{sec2.1}. Assume Hypothesis \ref{hypothesis decreasing intersection numbers}, and assume that we have chosen $\eta$ and $r$ as in  Theorem \ref{thm main topological}.  We note that given any path $\beta : [0, 1] \to \bar{Y}_{d + 1}$ and any loop $\gamma$ on the fiber of $\bar{Y}_{d + 1} \to Y_d$ containing $\beta(0)$, we can deform $\gamma$ along the path $\beta$ to a loop $\gamma'$ on the fiber of $\bar{Y}_{d + 1} \to Y_d$ containing $\beta(1)$. This follows from the fact that the family $\bar{Y}_{d + 1} \to Y_d$ enjoys the homotopy lifting extension property with respect to the inclusion $\{0, 1\} \hookrightarrow [0, 1]$ (see \cite[Proposition VII.6.3]{bredon1993topology}). Fix a path in $\bar{Y}_{d + 1}$ from $\infty_{T_{0}} = (a_{1}(z_{0}), ... , a_{d}(z_{0}), \infty) \in \bar{Y}_{d + 1}$ to $(1, ... , d, d + 1)$ so that each of the loops $\gamma_{I, n}$ is taken in this way to a loop in $(\proj_{\cc}^{1})^{\an} \smallsetminus \{1, ... , d\}$ which is homotopic to a circular loop surrounding the subinterval $I = \{m, ... , m + l - 1\} \subseteq \{1, ... , d\}$.  From now on, we identify $\pi_{1}((\proj_{\cc}^{1})^{\an} \smallsetminus \{1, ... , d\}, d + 1)$ with $\pi_{1}(\mathcal{F}_{z_{0}}, \infty_{z_0})$ via this path.

We will specify a generating set for the fundamental group 
$\pi_{1}((\proj_{\cc}^{1})^{\an} \smallsetminus \{1, ... , d\}, d + 1)$ by specifying one for $\pi_{1}((\mathbb{A}_{\cc}^{1})^{\an} \smallsetminus \{1, ... , d\}, d + 1)$ as follows. For $i = 1 , ..., d$ let $x_i$ be the loop on $(\aff_{\cc}^{1})^{\an}$ based at the point $d + 1$ going above the points $i + 1, ... , d$ and wrapping counterclockwise around only the point $i$. It is clear from the top short exact sequence in (\ref{eq S4}) that we may view these loops $x_i$ as elements of $P_{d+1} \lhd B_{d+1}$. It is easy to verify that as an element of $B_{d+1}$, each generator $x_i$ can be expressed  as $(\beta_{d}\cdots \beta_{i+1})\beta_{i}^{2}(\beta_{d}\cdots\beta_{i+1})^{-1}$ by checking that this braid is obtained by first moving the point $d + 1$ leftwards over the points $i + 1, ... , d$, then wrapping it counterclockwise around the point $i$, and finally returning it to its original position by performing the inverse of the first motion.

We see from the diagram in (\ref{eq S4}) (where we put $T_0 = (1, ... , d)$) that the action of $P_d$ on $\pi_{1}((\proj_{\cc}^{1})^{\an} \smallsetminus \{1, ..., d\}, d + 1)$ induced by the splitting of the bottom short exact sequence is given by the restriction of the conjugation action of $B_{d} \subset B_{d + 1}$ on the subgroup $\pi_{1}((\aff_{\cc}^{1})^{\an} \smallsetminus \{1, ..., d\}, d + 1) \subset B_{d + 1}$ (which is normalized by $B_{d}$). The following lemma, which describes this conjugation action, follows from straightforward calculations using the braid relations given in (\ref{braid relations}).

\begin{lemma} \label{lemma rho}

The elements $x_i$, viewed as braids in $B_{d+1}$, behave under conjugation by the braids $\beta_i \in B_{d} \subset B_{d + 1}$ as follows. 
$$\begin{cases}
\beta_i^{-1}x_{i+1}\beta_i = x_{i }, & \\
\beta_i^{-1}x_i\beta_i = \, x_ix_{i+1}x_i^{-1}, & \\
\beta_i^{-1}x_{j}\beta_i = x_{j}, & j \neq i, i + 1. 
\end{cases}$$

\end{lemma}


\begin{prop} \label{prop action of Dehn twists}

For any $(I, n) \in \mathcal{I}$ with $I = \{m, ... , m + l - 1\}$, the Dehn twist $D_{I, n} \in \pi_{0}\mathcal{Y}_{d}$ acts on $\pi_{1}((\proj_{\cc}^{1})^{\an} \smallsetminus \{1,...,d\}\},d+1)$ as 
$$\begin{cases}
x_{i} \mapsto (x_{m} \cdots x_{m+l-1})x_{i}(x_{m} \cdots x_{m+l-1})^{-1}, &  i \in I \\
x_{i} \mapsto x_{i}, &  i \notin I.
\end{cases}$$

\end{prop}

\begin{proof}

We first directly observe that the loop $\lambda_{I, n}\in P_d$ can be expressed in terms of the generators $\beta_{i}$ of $B_{d}$ as $(\beta_{m} \cdots \beta_{m+l-2})^{l} \in P_{d}$.  Therefore Lemma \ref{lemma braids mapping to Dehn twists} implies that $\rho((\beta_{m} \cdots \beta_{m+l-2})^{l}) = \varphi(D_{I, n})$, and so it will suffice to show, using Lemma \ref{lemma rho}, that $(\beta_{m} \cdots \beta_{m+l-2})^{l}$ acts on each $x_{i}$ in the manner described in the above statement.

It is straightforward to check that the the element $\beta_{m} \cdots \beta_{m+l-2}\in B_{d}$ acts on the generators $x_i$ by (right) conjugation as 
$$x_{m} \mapsto (x_{m} \cdots x_{m+l-2})x_{m+l-1}(x_{m} \cdots x_{m+l-2})^{-1} = (x_{m} \cdots x_{m+l-1})x_{m+l-1}(x_{m} \cdots x_{m+l-1})^{-1},$$
 as $x_i \mapsto x_{i - 1}$ for $m+1 \leq i \leq m + l - 1$, and as $x_i \mapsto x_i$ for $i \notin I$.  We easily deduce from these observations that conjugation by $\beta_{m} \cdots \beta_{m+l-2}$ fixes the product $x_{m} \cdots x_{m+l-1}$. This, combined with our formulas for how this conjugation acts on each $x_i$, shows that $(\beta_{m} \cdots \beta_{m+l-2})^{l}$ conjugates each $x_{i}$ by $(x_{m} \cdots x_{m+l-1})^{-1}$ for $i \in I$ while fixing each $x_{i}$ for $i \notin I$.

\end{proof}

\begin{rmk}\rm  \label{rmk conjugates}

In fact, it is not difficult to see from Proposition \ref{prop factors through MCG} and the formulas given in Lemma \ref{lemma rho} that more generally, each element of $\pi_{0}\mathcal{Y}_{d}$ acts on each generator $x_i$ by taking it to a conjugate of $x_i$.

\end{rmk}

\subsection{Explicit computations of prime-to-$p$ \'etale fundamental groups} \label{sec3.2}

Theorem \ref{thm comparison} (or rather, Remark \ref{rmk semidirect} below), now allows us to compute prime-to-$p$ fundamental groups of the sort $\pi_1^{\et}(\mathbb{P}^1_K \smallsetminus \{\alpha_1,...,\alpha_d\},P)^{(p')}$ with explicit generators and relations. We give two basic examples.

\begin{ex}

Given any elements $a_{1}, ... , a_{d} \in \mathbb{Z}_p^{\oper{un}} \cup \{\infty\}$ algebraic over $\qq$ that are pairwise distinct modulo $p$, we have (for any basepoint $P$) 
$$\pi_{1}^{\et}(\proj_{\qq_{p}^{\unr}}^{1} \smallsetminus \{a_1, ... , a_d\}, P)^{(p')} \cong \widehat{\langle x_1 , ..., x_d, \delta \ | \ x_1 \cdots x_d = 1, \, \{[\delta, x_i] = 1\}_{1 \leq i \leq d} \rangle}^{(p')}.$$

\end{ex}

\begin{ex}

Given any prime $p\geq 3$ and integer $m \geq 1$, we have (for any basepoint $P$) 
$$\pi_1^{\et}(\proj_{\qq_{p}^{\unr}}^{1} \smallsetminus \{0, p^m, 1, 2\}, P)^{(p')} \cong$$ 
$$\widehat{\langle x_1, x_2, x_3, x_4, \delta \ | \ x_1 \cdots x_4 = 1, \ \{\delta^{-1} x_i\delta= (x_1x_2)^{m} x_i (x_1x_2)^{-m}\}_{i = 1, 2},\ \{[\delta, x_{i}] = 1\}_{i = 3, 4} \rangle }^{(p')}.$$

\end{ex}

\subsection{The field of moduli of prime-to-$p$ covers} \label{sec3.3}

Recall that a given $G$-Galois branched cover of $\mathbb{P}^1_{\bar K}$ has a unique minimal field of definition as a Galois cover, namely its field of moduli over $K$ (as a $G$-Galois cover). Indeed, the obstruction for the field of moduli to be a field of definition is contained in $H^2(K,Z(G_K))$. (See, for example, \cite[\S1]{be}, \cite[Note below Theorem 1]{db1}, \cite[\S6.3]{db2}.) Since $K$ is a strictly Henselian field, this cohomology group vanishes.

Combining the explicit algorithm in \S\ref{sec3.1} with Theorem \ref{thm comparison} gives us the following corollary about the field of moduli of prime-to-$p$ Galois branched covers of $\mathbb{P}^1_{\bar K}$.

\begin{cor} \label{prop field of def}

Let $G$ be a finite prime-to-$p$ group, and let $Y \to \proj_{\bar{K}}^{1}$ be a $G$-Galois branched cover which is ramified only over $K$-rational points. Then the degree of its field of moduli (as a $G$-Galois cover) over $K$ divides the exponent of the quotient $G / Z(G)$.

\end{cor}

\begin{proof}

Let $\alpha_{1}, ... , \alpha_{d}$ be the $K$-points of $\proj_{K}^{1}$ over which the map $Y \to \proj_{\bar K}^{1}$ is ramified, so that this cover corresponds to a surjection $\phi: \pi_{1}^{\et}(\proj_{\bar{K}}^{1} \smallsetminus \{\alpha_{1}, ... , \alpha_{d}\}, P) \twoheadrightarrow G$ (where $P$ is some basepoint). Let $a_{1}, ... , a_{d} \in \cc[[x]]$ be power series satisfying Hypothesis \ref{hypothesis varepsilon} which have the same intersection behavior as the $\alpha_{i}$'s. Let $x_1, ... , x_d$ be the set of generators of $\pi_1((\proj_{\cc}^{1})^{\an}\smallsetminus T_0,\infty_{z_0})$ specified in \S\ref{sec3.1}, and let $\bar{x}_1, ... , \bar{x}_d$ be their images under the isomorphism from  $\hat\pi_{1}((\proj_{\cc}^{1})^{\an} \smallsetminus T_{0}, \infty)^{(p')}$ to $\pi_{1}^{\et}(\proj_{\bar{K}}^{1} \smallsetminus \{\alpha_{1}, ... , \alpha_{d}\}, P)^{(p')}$ guaranteed by Theorem \ref{thm comparison}.  It follows from Proposition \ref{prop action of Dehn twists} (or Remark \ref{rmk conjugates}),  Theorem \ref{thm main topological}, and Theorem \ref{thm comparison} that up to inner automorphism, a topological generator $\bar{\delta}$ of $G_{K}^{(p')}$ acts by taking each $\bar{x}_{i}$ to $\bar{y}_{i}^{-1}\bar{x}_{i}\bar{y}_{i}$ for some $\bar{y}_{i} \in \pi_{1}^{\et}(\proj_{\bar{K}}^{1} \smallsetminus \{\alpha_{1}, ... , \alpha_{d}\})^{(p')}$.  Let $N$ denote the exponent of $G / Z(G)$.  Then $\bar{\delta}^{N}$ acts up to inner automorphism as $\bar{x}_{i} \mapsto \bar{y}_{i}^{-N}\bar{x}_{i}\bar{y}_{i}^{N}$, so for $1 \leq i \leq d$, we have that the elements $\phi(\bar{x}_{i}^{\bar{\delta}^{N}})$ are uniformly conjugate to the elements $\phi(\bar{y}_{i})^{-N}\phi(\bar{x}_{i})\phi(\bar{y}_{i})^{N} = \phi(\bar{x}_{i})$, where the last equality follows from the fact that the $N^{\oper{th}}$ power of any element of $G$ lies in $Z(G)$. It follows that the field of moduli (as a $G$-Galois cover) of $Y\rightarrow \mathbb{P}^1_{\bar K}$ is contained in the fixed field of $\bar{\delta}^{N}$, which is $K(\pi^{1/N})$.

\end{proof}

\begin{rmk} \label{rmk field of def} 

With some work one can show that Th\'eor\`emes 3.2 and 3.7 in \cite{saidi}, adapted to our situation, imply that $K(t^{1/M})$ is a field of definition of the cover (together with its Galois action), where $M$ is the exponent of $G$. The corollary above is a strengthening of this result.


\end{rmk}

\section{Proof of the comparison theorem} \label{sec4}

We first observe that the arithmetic part of the statement of Theorem \ref{thm comparison}(a) is an immediate corollary of the ideas in \cite{kisin2000prime}. In fact, slightly more is true.

\begin{prop}\label{prop kisin}
Theorem \ref{thm comparison}(a) holds for $\rho_{\alg}^{(p')}$. Furthermore, if $R = \cc[[x]]$, $\alpha_i \in \cc[x]$ for $1 \leq i \leq d$, and $q$ is any prime, then 
$\rho_{\alg}^{(q')}$ factors through $G_K^{(q')}$.
\end{prop}

\begin{proof}
The statement for $\rho_{\alg}^{(p')}$ follows from \cite[Corollary 1.16]{kisin2000prime}, which is stated in terms of outer Galois actions but easily implies the statement for full Galois actions.

In order to prove the statement for $\rho_{\alg}^{(q')}$ when $K = \cc((x))$, we first observe that the statement of \cite[Corollary 5.5.8]{szamuely2009galois}, which gives a condition for the injectivity of maps of fundamental groups, works for any Galois category (it is only stated there for the Galois category of finite \'{e}tale covers).  Therefore, it suffices in this case to prove that given any finite prime-to-$p$ group $G$, every $G$-Galois cover of $\proj_{\bar{K}}^{1} \smallsetminus \{\alpha_{1}, ... , \alpha_{d}\}$ has a model defined over a prime-to-$p$ extension of $K$.  This follows from Proposition \ref{prop field of def} (or from Remark \ref{rmk field of def}); here we note that the case that we require of Proposition \ref{prop field of def} relies only on Proposition \ref{prop isomorphic to family over cc((x))}, which is independent from the rest of the paper.

\end{proof}

We will now justify a series of simplifying assumptions. We will finally assume that $R=\mathbb{C}[[x]]$, in which case $p = 0$ and therefore Theorem \ref{thm comparison}(a) holds trivially.

\subsection{Reduction to the case that $P=\infty$ and $\alpha_i \in R$} \label{sec4.1}

It is clear that the basepoint $P$ can be chosen arbitrarily, and so it suffices to prove the following.

\begin{lemma}

It suffices to prove the statement of Theorem \ref{thm comparison} in the case the $\alpha_{i} \in R$ for $1 \leq i \leq d$.

\end{lemma}

\begin{proof}
We first observe that every $K$-automorphism $\phi$ of $\proj_{K}^{1}$ induces an isomorphism $$\pi_{1}^{\et}(\proj_{\bar{K}}^{1} 
\smallsetminus \{\alpha_{1}, ... , \alpha_{d}\}, \phi^{-1}(\infty)) \stackrel{\sim}{\to} \pi_{1}^{\et}(\proj_{\bar{K}}^{1} \smallsetminus 
\{\phi(\alpha_{1}), ... , \phi(\alpha_{d})\}, \infty)$$
which respects the action of $G_{K}$.

Choose $\beta \in R^{\times}$ such that $\alpha_i - \beta \in R^{\times}$ for $1 \leq i \leq d$. (This is always possible because the residue field of $R$ is infinite.) Then it is easy to check that the automorphism $\phi$ given by $z \mapsto z / (z-\beta)$ respects the intersection pairing. Thus, we may move any of the $\alpha_{i}$'s away from infinity, so assume that $\alpha_{i} \in K$ for $1 \leq i \leq d$.

Note that we may apply some power of the automorphism $\phi : z \mapsto \pi z$ to move the $\alpha_{i}$'s to elements of $R$, and that in general, this $\phi$ will change the intersection behavior of the $\alpha_i$'s. Let $E'_{i,j}$ be the intersection index of $\pi \alpha_i$ with $\pi \alpha_j$. Let $a'_{1},...,a'_d$ be elements of $\mathbb{C}[[x]]$ with intersection indices $E'_{i,j}$, and assume without loss of generality that the $a_i$'s as well as the $a'_i$'s satisfy Hypothesis \ref{hypothesis varepsilon} for $\varepsilon$ and that $z_0$ is chosen so that Theorem \ref{thm main topological} will hold for both the $a_i$'s and the $a'_i$'s. By Remark \ref{rmk intersection behavior}(a), we may assume without loss of generality that $a'_{i}(z_{0}) = a_{i}(z_{0})$ for $1 \leq i \leq d$. Define the action $(\bar{\rho}'_{\tp})^{(p')}$ for the $a'_i$'s in the same way that $\bar{\rho}_{\tp}^{(p')}$ was defined for the $a_i$'s. Then it follows from our observation at the beginning of the proof together with the statement of Theorem \ref{thm comparison} for the case of points in $R$ implies that the action $(\bar{\rho}'_{\tp})^{(p')}$ factors through $G_{K}^{(p')}$ and is isomorphic to $\bar{\rho}_{\alg}^{(p')}$. It will therefore suffice to show that the outer monodromy actions $\bar{\rho}_{\tp}$ and $\bar{\rho}'_{\tp}$ are isomorphic.

Define $\mathcal{I}$ as in \S\ref{sec2}, and define $\mathcal{I}'$ analogously for the power series $a'_i$; similarly, define the loops $\gamma_{I, n}$ for $(I, n) \in \mathcal{I}$ as in \S\ref{sec2} and define the loops $\gamma'_{I', n'}$ for $(I', n') \in \mathcal{I}$ analogously. Assume that the $\alpha_i$'s are ordered so that for some $l$, we have that $\alpha_i \notin R$ if and only if $i \in \{l + 1, ... , d\}$. Then it is clear that there is a bijection $\Phi: \mathcal{I} \to \mathcal{I'}$ given by $(I, n) \mapsto (I, n + 1)$ if $I \subseteq \{1, ... , l\}$, $(I, n) \mapsto (I, n - 1)$ if $I \subseteq \{l + 1, ... , d\}$ and $n \geq 2$, and $(\{l + 1, ... , d\}, 1) \mapsto (\{1, ... , l\}, 1)$. Note moreover that for each $(I, n) \in \mathcal{I}$, the loop $\im(\gamma_{I, n})$ separates the subset of points $\{a_{i}(z_{0})\}_{i \in I}$ from its complement in $\{a_{i}(z_{0})\}_{i = 1}^{d}$ in $\proj_{\cc}^{1}$ and that the same statement holds for each $(I', n') \in \mathcal{I}'$. It is easy to see from this that the loop $\gamma_{I,n}$ is homotopic to the the loop $\gamma'_{\Phi((I,n))}$ on $(\mathbb{P}^1_{\mathbb{C}})^{\an} \smallsetminus\{a_1(z_0),...,a_d(z_0)\}$. Therefore the Dehn twists $D_{I,n}$ and $D_{\Phi((I,n))}$ act on $\pi_{1}(\mathcal{F}_{z_{0}}, \infty)$ in the same way up to inner automorphism. Now it follows from the description of the monodromy action in Theorem \ref{thm main topological} that the actions $\bar{\rho}_{\tp}$ and $\bar{\rho}'_{\tp}$ are isomorphic, and we are done.

\end{proof}

\subsection{Reduction to the case that $R$ is complete and embeds in $\cc$} \label{sec4.2}
We will assume for the remainder of this paper that $P = \infty$ and that $\alpha_i \in R$ for $1 \leq i \leq d$.

\begin{lemma} \label{prop factoring through prime-to-p}

The sequence of morphisms $\proj_{\bar{K}}^{1} \smallsetminus \{\alpha_{1}, ... , \alpha_{d}\} \to \proj_{K}^{1} \smallsetminus \{\alpha_{1}, ... , \alpha_{d}\} \to \Spec(K)$ induces a split short exact sequence 
\begin{equation} \label{short exact sequence of prime-to-p algebraic} 1 \to \pi_{1}^{\et}(\proj_{\bar{K}} \smallsetminus \{\alpha_{1}, ... , \alpha_{d}\}, P)^{(p')} \to \pi_{1}^{\et}(\proj_{K}^{1} \smallsetminus \{\alpha_{1}, ... , \alpha_{d}\}, P)^{(p')} \to G_{K}^{(p')} \to 1. \end{equation}
Furthermore, if $K = \cc[[x]]$ and $\alpha_i \in \cc[x]$ for $1 \leq i \leq d$, then for any prime $q$, we have
\begin{equation} \label{short exact sequence of prime-to-p algebraic two} 1 \to \pi_{1}^{\et}(\proj_{\bar{K}} \smallsetminus \{\alpha_{1}, ... , \alpha_{d}\}, P)^{(q')} \to \pi_{1}^{\et}(\proj_{K}^{1} \smallsetminus \{\alpha_{1}, ... , \alpha_{d}\}, P)^{(q')} \to G_{K}^{(q')} \to 1. \end{equation}

\end{lemma}

\begin{proof}

Exactness on the right is immediate from \cite[Proposition 3.2.5]{zalesskii}. Exactness on the left follows from Proposition \ref{prop kisin}, using \cite[Proposition 4]{leftexact}.

\end{proof}

\begin{lemma} \label{lemma assumptions on R}
 
There is a strictly Henselian DVR $R'$ with the following properties:

a) $R'$ is the completion of a countable subring of $R$ with uniformizer $\pi$ and which contains $\alpha_{1}, ... , \alpha_{d}$; and 

b) letting $K'$ denote the fraction field of $R'$ and $(\rho_{\alg}')^{(p')} : G_{K'}^{(p')} \to \Aut(\pi_{1}^{\et}(\proj_{\bar{K'}}^{1} \smallsetminus \{\alpha_{1}, ... , \alpha_{d}\}, \infty)^{(p')})$ be the Galois action analogous to $\rho_{\alg}^{(p')}$, there are isomorphisms $G_{K}^{(p')} \stackrel{\sim}{\to} G_{K'}^{(p')}$ and $\pi_{1}^{\et}(\proj_{\bar{K}}^{1} \smallsetminus \{\alpha_{1}, ... , \alpha_{d}\}, \infty)^{(p')} \stackrel{\sim}{\to} \pi_{1}^{\et}(\proj_{\bar{K'}}^{1} \smallsetminus \{\alpha_{1}, ... , \alpha_{d}\}, \infty)^{(p')}$ inducing an isomorphism of the actions $\rho_{\alg}^{(p')}$ and $(\rho_{\alg}')^{(p')}$.

\end{lemma}

\begin{proof}

Let $R''$ be the integral closure of $R \cap \qq(\pi,\alpha_{1}, ... , \alpha_{d})$ in $R$, and let $K''$ denote the fraction field of $R''$.  Then $R''$ contains $\pi$ and we have $\pi R \cap R'' = \pi R''$ because $\pi R \subset R$ is prime and every element in $R'' \smallsetminus \pi R''$ clearly has an inverse in $R''$.  The strict Henselization of the localization of $R''$ at $\pi R''$ is countable by construction; let $R'$ be its completion.  Then clearly $R'$ satisfies the properties stated in (i).  By Lemma \ref{prop factoring through prime-to-p}, there are split short exact sequences of prime-to-$p$ quotients of \'{e}tale fundamental groups associated to the projective line minus the points $\alpha_{1}, ... , \alpha_{d}$ over the schemes $\Spec(K)$, $\Spec(K'')$, and $\Spec(K')$.  We consider the commutative diagram below, where the rows are these split short exact sequences and the vertical arrows are induced by the field inclusions $\overline{K''} \subseteq \bar{K}$ and $\overline{K''} \subset \overline{K'}$.
\begin{equation} \label{commutative diagram assumptions on R} \xymatrix{ 1 \ar[r] & \pi_{1}^{\et}(\proj_{\bar{K}}^{1} \smallsetminus \{\alpha_{1}, ... , \alpha_{d}\}, \infty)^{(p')} \ar[r] \ar[d] & \pi_{1}^{\et}(\proj_{K}^{1} \smallsetminus \{\alpha_{1}, ... , \alpha_{d}\}, \infty)^{(p')} \ar[r] \ar[d] & G_{K}^{(p')} \ar@/_1pc/[l] \ar[r] \ar[d] & 1 
\\ 1 \ar[r] & \pi_{1}^{\et}(\proj_{\overline{K''}} \smallsetminus \{\alpha_{1}, ... , \alpha_{d}\}, \infty)^{(p')} \ar[r] & \pi_{1}^{\et}(\proj_{K''}^{1} \smallsetminus \{\alpha_{1}, ... , \alpha_{d}\}, \infty)^{(p')} \ar[r] & G_{K''}^{(p')} \ar@/_1pc/[l] \ar[r] & 1
\\ 1 \ar[r] & \pi_{1}^{\et}(\proj_{\overline{K'}}^{1} \smallsetminus \{\alpha_{1}, ... , \alpha_{d}\}, \infty)^{(p')} \ar[r] \ar[u] & \pi_{1}^{\et}(\proj_{K'}^{1} \smallsetminus \{\alpha_{1}, ... , \alpha_{d}\}, \infty)^{(p')} \ar[r] \ar[u] & G_{K'}^{(p')} \ar@/_1pc/[l] \ar[r] \ar[u] & 1 } \end{equation}
Clearly the vertical arrows on the left are isomorphisms since they are induced by inclusions of algebraically closed fields. Note that it follows from a special case of Abhyankar's Lemma that the prime-to-$p$ absolute Galois group of any strictly Henselian field of residue characteristic $p$ is isomorphic to $\widehat{\zz}^{(p')}$. Therefore, the vertical arrows on the right are also isomorphisms.  Then the vertical arrows in the middle are also isomorphisms by the Five Lemma.  Since $\rho_{\alg}^{(p')}$ and $(\rho_{\alg}')^{(p')}$ are induced by the splittings of the top and bottom rows respectively, we are done.

\end{proof}

\subsection{Reduction to the case $R = \cc[[x]]$} \label{sec4.3}

In order to reduce to the case of $R = \cc[[x]]$, we will construct a ring (which we will call $S$) lying inside $\cc[[x]]$ which has $K$ as a quotient. This general strategy was inspired by a similar one used in a different context in \cite[\S5]{oda1995note}.

We fix, once and for all, an embedding $R \hookrightarrow \cc$.  Let $S = R[[x]][\frac{1}{x}]$, and write $F$ for its fraction field.  The embedding $R \hookrightarrow \cc$ determines an inclusion of $F$ into the field $\cc((x))$; let $\bar{F}$ be the algebraic closure of $F$ inside $\overline{\cc((x))}$.  It is easy to verify that there is a unique  $R$-algebra surjection $\psi : S \twoheadrightarrow K$, continuous in the $x$-adic topology, which sends $x$ to $\pi$.

We now construct elements $\tilde{\alpha}_i \in R[x] \subset S$ with the same  intersection behavior as the $\alpha_i$'s and satisfying $\psi(\tilde{\alpha}_i) = \alpha$.  We observe that for any elements $\tilde{\alpha}_1, ... , \tilde{\alpha}_d \in S$ with $\psi(\tilde{\alpha}_i) = \alpha_i$ for $1 \leq i \leq d$, we have $e_{i, j} \leq E_{i, j}$ for $1 \leq i < j \leq d$, where $e_{i, j}$ is the $x$-adic valuation of  $\tilde{\alpha}_i-\tilde{\alpha}_j$.  Given such elements $\tilde{\alpha}_i \in R[x]$, we make adjustments to these polynomials in order to ensure that $e_{i, j} = E_{i, j}$ for $1 \leq i < j \leq d$, in the following manner.  Suppose that $m := e_{1, 2} < E_{1, 2}$.  Let $b \in R$ denote the coefficient of $x^{m}$ in $\tilde{\alpha}_2 - \tilde{\alpha}_1$; it is clear from the fact that $\pi^{m + 1}$ divides $\alpha_2 - \alpha_1$ that $b$ is divisible by $\pi$.  Then add the element $\pi^{-1}bx^{m}(x - \pi) \in \ker(\psi)$ to $\tilde{\alpha}_2$, so that we still have $\psi(\tilde{\alpha}_2) = \alpha_2$ but now $m + 1 \leq e_{1, 2} \leq E_{1, 2}$.  After repeating this process a finite number of times, we get $e_{1, 2} = E_{1, 2}$.  Therefore, we may start by setting $\tilde{\alpha}_i = \alpha_i \in R \subset R[x]$ for $1\leq i \leq d$, and then for $j$ running through $\{2,...,d\}$, and for $i$ running through $\{1,...,j-1\}$, follow this procedure to ensure that $e_{i, j} = E_{i, j}$. It is easy to verify that this gives us polynomials $\tilde{\alpha}_i \in R[x]$ with the desired properties. 

\begin{lemma} \label{lemma etale divisor}

The divisor on the surface $\proj_{S}^{1}$ given by the formal sum $\sum_{i = 1}^{d} (\tilde{\alpha}_{i})$ is \'{e}tale over $\Spec(S)$.

\end{lemma}

\begin{proof}

It suffices to show that the map $\bigcup_{i = 1}^{d} (\tilde{\alpha}_{i}) \to \Spec(S)$ is unramified, or equivalently, that for $i \neq j$ and for each prime $\mathfrak{p}$ of $S$, the images of $\tilde{\alpha}_{i}$ and $\tilde{\alpha}_{j}$ modulo $\mathfrak{p}$ are distinct.  We claim that any prime of $S$ is either $(\pi)$ or of the form $(g)$, where $g \in R[x]$ is an irreducible monic polynomial all of whose non-leading coefficients lie in $(\pi)$ (we call any monic polynomial with this property a \textit{distinguished polynomial}).  Let $\mathfrak{p}$ be any prime ideal of $S$ different from $(\pi)$.  The $p$-adic Weierstrass preparation theorem says that each element of $R[[x]]$ can be written as the product $u\pi^{m}g$, where $u \in R[[x]]^{\times}$ is a unit, $m \geq 0$ is an integer, and $g$ is a distinguished polynomial.  Therefore, the prime $\mathfrak{p} \neq (\pi)$ of $S$ is generated by elements $g_{1}, \pi^{m_{2}}g_{2}, ... , \pi^{m_{r}}g_{r}$, where $m_{l} \geq 0$ is an integer for $2 \leq l \leq r$ and $g_{l} \in R[x]$ is distinguished for $1 \leq l \leq r$.  One checks using Gauss' Lemma that the greatest common divisor $g$ of the $g_{l}$'s is also a distinguished polynomial in $R[x]$.  Clearly $\mathfrak{p}$ contains $\pi^{m}g$, where $m$ is the maximum of the $m_{l}$'s.  Since $\mathfrak{p}$ is prime, we have $\pi \in \mathfrak{p}$ or $g \in \mathfrak{p}$.  If $\pi \in \mathfrak{p}$, then since $\pi$ divides the nonleading terms of the polynomial $g$, some power of $x$ is contained in $\mathfrak{p}$, which contradicts the fact that $x$ is a unit in $S$.  It follows that $g \in \mathfrak{p}$ and in fact that $\mathfrak{p} = (g)$ with $g$ irreducible, proving the above claim.

It now suffices to show that for $i \neq j$, the polynomials $\tilde{\alpha}_{i}$ and $\tilde{\alpha}_{j}$ are distinct modulo $(\pi)$ and modulo $(g)$ for any irreducible distinguished $g \in R[x] \smallsetminus \{x\}$.  Let $m = e_{i, j} = E_{i, j}$, so that the $m$th coefficients $b_{i, m}$ and $b_{j, m}$ of $\tilde{\alpha}_{i}$ and $\tilde{\alpha}_{j}$ are distinct.  Then, modulo $\pi$, we have $\psi(b_{j, m} - b_{i, m}) \equiv \psi(x^{-m}(\tilde{\alpha}_j - \tilde{\alpha}_i)) = \pi^{-m} (\alpha_j - \alpha_i) \in R^{\times}$, so we have $b_{j, m} - b_{i, m} \in R^{\times}$.  In particular, this shows that $\tilde{\alpha}_{i}$ and $\tilde{\alpha}_{j}$ are distinct modulo $(\pi)$.  Now clearly, $S / (g) = K(\beta)$ where $\beta \in \bar{K}$ is a root of the polynomial $g$; note that $\beta$ has positive valuation $s > 0$ due to the fact that $g$ is distinguished.  The images of $\tilde{\alpha}_{i}$ and $\tilde{\alpha}_{j}$ are given by the polynomials $\tilde{\alpha}_{i}(\beta)$ and $\tilde{\alpha}_{j}(\beta)$ in $K(\beta)$.  Now $\tilde{\alpha}_{i}(\beta) - \tilde{\alpha}_{j}(\beta)$ can be written as a polynomial in $\beta$ whose lowest-order term is $(b_{i, m} - b_{j, m})\beta^{m}$, which has valuation $ms$ and is therefore nonzero.  Thus, we have $\tilde{\alpha}_i(\beta) \neq \tilde{\alpha}_j(\beta)$, as desired.

\end{proof}

We will now show that the prime-to-$p$ Galois action associated to our scheme over $\cc((x))$ is isomorphic to the one associated to our scheme over $K$.

\begin{prop} \label{prop reducing to K}

With the above assumptions and notation, define $\rho_{0}^{(p')} : G_{\cc((x))} \to \Aut(\pi_{1}^{\et}(\proj_{\overline{\cc((x))}}^{1} \smallsetminus \{\tilde{\alpha}_{1}, ... , \tilde{\alpha}_{d}\}, \infty)^{(p')})$ to be the action induced by the action of $G_{\cc((x))}$ on $\pi_{1}^{\et}(\proj_{\overline{\cc((x))}}^{1} \smallsetminus \{\tilde{\alpha}_{1}, ... , \tilde{\alpha}_{d}\}, \infty)$ determined by the splitting of the fundamental short exact sequence 
\begin{equation} \label{short exact sequence of algebraic fundamental groups cc((x))} 1 \to \pi_{1}^{\et}(\proj_{\overline{\cc((x))}}^{1} \smallsetminus \{\tilde{\alpha}_{1}, ... , \tilde{\alpha}_{d}\}, \infty) \to \pi_{1}^{\et}(\proj_{\cc((x))}^{1} \smallsetminus \{\tilde{\alpha}_{1}, ... , \tilde{\alpha}_{d}\}, \infty) \to G_{\cc((x))} \to 1 \end{equation}
 induced by the point at infinity.  Then $\rho_{0}^{(p')}$ factors through $G_{\cc((x))}^{(p')}$, and we have isomorphisms $G_{K}^{(p')} \stackrel{\sim}{\to} G_{\cc((x))}^{(p')}$ and $\pi_{1}^{\et}(\proj_{\bar{K}}^{1} \smallsetminus \{\alpha_{1}, ... , \alpha_{d}\}, \infty)^{(p')} \stackrel{\sim}{\to} {\pi}_{1}^{\et}(\proj_{\overline{\cc((x))}}^{1} \smallsetminus \{\tilde{\alpha}_{1}, ... , \tilde{\alpha}_{d}\}, \infty)^{(p')}$ inducing an isomorphism of the actions $\rho_{\alg}^{(p')}$ and $\rho_{0}^{(p')}$.

\end{prop}

\begin{proof}

By Proposition \ref{prop kisin} applied to the fundamental short exact sequences associated to the families over both $K$ and $\cc((x))$, we have that $\rho_{0}^{(p')}$ and $\rho_{\alg}^{(p')}$ factor through $G_{\cc((x))}^{(p')}$ and $G_{K}^{(p')}$ respectively. By Lemma \ref{prop factoring through prime-to-p}, taking prime-to-$p$ quotients of the terms in the fundamental short exact sequences corresponding to each yields split short exact sequences of prime-to-$p$ quotients inducing $\rho_{0}^{(p')}$ and $\rho_{\alg}^{(p')}$ respectively.  Therefore, it suffices to show that these two split short exact sequences of prime-to-$p$ quotients are isomorphic.  We will do so by constructing intermediate short exact sequences of prime-to-$p$ quotients of \'{e}tale fundamental groups, associated to a family of $S$-schemes, which are isomorphic to both.

Consider the family $\mathfrak{F}_{S} := \proj_{S}^{1} \smallsetminus \{\tilde{\alpha}_{1}, ... , \tilde{\alpha}_{d}\} \to \Spec(S)$.  We write $\mathfrak{F}_{F}$, $\mathfrak{F}_{\cc((x))}$, and $\mathfrak{F}_{K}$ for the base changes of $\mathfrak{F}_{S}$ with respect to the inclusions $S \hookrightarrow F$ and $S \hookrightarrow \cc((x))$ and the surjection $\psi: S \twoheadrightarrow K$ respectively, and we write $\mathfrak{F}_{\bar{F}}$, etc. for their geometric fibers.  Note in particular that $\mathfrak{F}_{K} = \proj_{K}^{1} \smallsetminus \{\alpha_{1}, ... , \alpha_{d}\}$.  Let $\bar{\eta}$ be a generic geometric point over the generic point $\eta \in \Spec(S)$, and let $\bar{s}$ be a geometric point over the closed point $s := (x - \pi) \in \Spec(S)$.  Since the divisor given by the points $\tilde{\alpha}_{i} \in \proj_{S}^{1}$ is \'{e}tale over $\Spec(S)$ by Lemma \ref{lemma etale divisor}, we can apply \cite[Proposition XIII.4.3 and Exemples XIII.4.4]{grothendieck224revetements}, which shows that the sequence of morphisms $\mathfrak{F}_{\bar{K}} \to \mathfrak{F}_{S} \to \Spec(S)$ along with the section at infinity induces a split short exact sequence 
\begin{equation} \label{short exact sequence Grothendieck} 1 \to \pi_{1}^{\et}(\mathfrak{F}_{\bar{K}}, \infty)^{(p')} \to \pi_{1}^{\et}(\mathfrak{F}_{S}, \infty_{\bar{s}})' \to \pi_{1}^{\et}(\Spec(S), \bar{s}) \to 1. \end{equation}
Here $\pi_{1}^{\et}(\mathfrak{F}_{S}, \infty)'$ is the quotient of $\pi_{1}^{\et}(\mathfrak{F}_{S}, \infty)$ constructed in \cite[\S XIII.4]{grothendieck224revetements}.  In fact, the generalized version of Abhyankar's Lemma appearing as \cite[Theorem 2.3.2]{grothendieck1971tame}, along with the fact that there are no \'{e}tale covers of $\Spec(S)$ of degree divisible by $p$, implies that the only \'{e}tale covers of $\Spec(S)$ are induced by adjoining prime-to-$p$ roots of $x$ to $S$.  Thus, we have $\pi_{1}^{\et}(\Spec(S), \overline{s}) = \pi_{1}^{\et}(\Spec(S), \overline{s})^{(p')} \cong \hat{\zz}^{(p')}$.  It follows that $\pi_{1}^{\et}(\mathfrak{F}_{S}, \infty)' = \pi_{1}^{\et}(\mathfrak{F}_{S}, \infty)^{(p')}$, so (\ref{short exact sequence Grothendieck}) is actually a short exact sequence of prime-to-$p$ quotients of \'{e}tale fundamental groups.

Similarly, we have a split short exact sequence 
\begin{equation} \label{exact sequence of prime-to-p generic} 1 \to \pi_{1}^{\et}(\mathfrak{F}_{\bar{F}}, \infty)^{(p')} \to \pi_{1}^{\et}(\mathfrak{F}_{S}, \infty_{\bar{\eta}})^{(p')} \to \pi_{1}^{\et}(\Spec(S), \bar{\eta})^{(p')} \to 1. \end{equation}
Choose compatible change-of basepoint isomorphisms $\pi_{1}^{\et}(\mathfrak{F}_{S}, \infty_{\bar{\eta}})^{(p')} \stackrel{\sim}{\to} \pi_{1}^{\et}(\mathfrak{F}_{S}, \infty_{\bar{s}})^{(p')}$ and 

\noindent $\pi_{1}^{\et}(\Spec(S), \bar{\eta})^{(p')} \stackrel{\sim}{\to} \pi_{1}^{\et}(\Spec(S), \bar{s})^{(p')}$.  Since the divisor given by the points $\tilde{\alpha}_{i} \in \proj_{S}^{1}$ is \'{e}tale over $\Spec(S)$ by Lemma \ref{lemma etale divisor}, we can apply a variant of Grothendieck's Specialization Theorem (\cite[Th\'{e}or\`{e}me 4.4]{orgogozo2000theoreme}) to obtain an isomorphism $\mathrm{sp} : \pi_{1}^{\et}(\mathfrak{F}_{\bar{F}}, \infty)^{(p')} \stackrel{\sim}{\to} \pi_{1}^{\et}(\mathfrak{F}_{\bar{K}}, \infty)^{(p')}$.  It is easy to check that $\mathrm{sp}$ commutes with the horizontal arrows in (\ref{exact sequence of prime-to-p generic}) and (\ref{short exact sequence Grothendieck}) and the change-of-basepoint isomorphisms.  It follows that the sequence in (\ref{exact sequence of prime-to-p generic}) is isomorphic as a split short exact sequence to the one in (\ref{short exact sequence Grothendieck}).

Now these isomorphic exact sequences induced by the family over $S$ fit into the commutative diagram below.

\begin{equation} \label{commutative diagram reducing to K} \xymatrix{ 1 \ar[r] & \pi_{1}^{\et}(\mathfrak{F}_{\bar{K}}, \infty)^{(p')} \ar[r] 
\ar@{=}[d] & \pi_{1}(\mathfrak{F}_{K}, \infty)^{(p')} \ar[r] \ar[d] & G_{K}^{(p')} \ar@/_1pc/[l] \ar[r] \ar[d]^{\wr} & 1
\\ 1 \ar[r] & \pi_{1}^{\et}(\mathfrak{F}_{\bar{K}}, \infty)^{(p')} \ar[r] & \pi_{1}^{\et}(\mathfrak{F}_{S}, \infty_{\bar{s}})^{(p')} \ar[r] & \pi_{1}^{\et}(\Spec(S), \bar{s})^{(p')} \ar@/_1pc/[l] \ar[r] & 1 
\\ 1 \ar[r] & \pi_{1}^{\et}(\mathfrak{F}_{\bar{F}}, \infty)^{(p')} \ar[r] \ar[u]^{\mathrm{sp} \wr} & \pi_{1}^{\et}(\mathfrak{F}_{S}, \infty_{\bar{\eta}})^{(p')} \ar[r] \ar[u]^{\wr} & \pi_{1}^{\et}(\Spec(S), \bar{\eta})^{(p')} \ar@/_1pc/[l] \ar[r] \ar[u]^{\wr} & 1
\\ 1 \ar[r] & \pi_{1}^{\et}(\mathfrak{F}_{\overline{\cc((x))}}, \infty)^{(p')} \ar[r] \ar[u]^{\wr} & \pi_{1}^{\et}(\mathfrak{F}_{\cc((x))}, \infty)^{(p')} \ar[r] \ar[u] & G_{\cc((x))}^{(p')} \ar@/_1pc/[l] \ar[r] \ar[u]^{\wr} & 1 } \end{equation}
Here the vertical maps from the terms in the first row to those in the second row are induced by the surjection $\psi : S \twoheadrightarrow K$, and the vertical maps from the terms in the last row to those in the third row are induced by the inclusion $S \hookrightarrow \cc((x))$ given above.  The bottom vertical arrows on the left is an isomorphism because $\bar{F} \hookrightarrow \overline{\cc((x))}$ is an inclusion of algebraically closed fields, and the top and bottom vertical arrows on the right are isomorphisms because all of these groups are isomorphic to $\hat{\zz}^{(p')}$, as noted above.  Therefore, the split short exact sequences at the top and bottom are isomorphic.  Since the actions $\rho_{\alg}^{(p')}$ and $\rho_{0}^{(p')}$ are induced by these sequences, we get the claimed isomorphism between $\rho_{\alg}^{(p')}$ and $\rho_{0}^{(p')}$.

\end{proof}

\begin{cor}

In order to prove Theorem \ref{thm comparison}, it suffices to prove the statement for $R = \cc[[x]]$, with the assumption that the elements $\alpha_i$ are polynomials.

\end{cor}

\begin{proof}

Since $\rho_{0}^{(p')}$ factors through $G_{\cc((x))}^{(p')}$ by Proposition \ref{prop factoring through prime-to-p}, it follows from the statement of Theorem \ref{thm comparison} over $\cc((x))$ that $\rho_{\tp}^{(p')}$ factors through $\pi_1(B_{\varepsilon}^{*},z_0)^{(p')}$ and that the actions $\rho_{0}^{(p')}$ and $\rho_{\tp}^{(p')}$ are isomorphic. The desired statement then follows from Proposition \ref{prop reducing to K}.

\end{proof}

\begin{rmk}

In fact, it is possible to generalize these results, and hence the statement of Theorem \ref{thm comparison}, to the case that $R$ is a strictly Henselian DVR of characteristic $p > 0$, at least under the assumption that its residue field is $\bar{\ff}_p$. Suppose that $R$ is a (strictly) Henselian DVR with residue field $\bar{\ff}_p$ and fraction field $K$ (e.g. $K$ is a strict Henselization of a global field of characteristic $p$ localized at some prime). The arguments in \S\ref{sec4.1} and \S\ref{sec4.2} proceed more or less in the same way that they did in the case of characteristic $0$. To prove Proposition \ref{prop reducing to K}, one may use a variant of the construction and arguments developed in this section, which we outline below.

We first note that a given uniformizer $\pi$ is transcendental over $\bar{\ff}_p \subset R$, which implies (after applying the arguments in \S\ref{sec4.2} and replacing $R$ with its completion) that $R = \bar{\ff}_p[[\pi]]$. Let $W(\bar{\ff}_p)$ denote the Witt ring of $\bar{\ff}_p$; let $x$ be transcendental over $W(\bar{\ff}_p)$; and let $S = W(\bar{\ff}_p)[[x]][\frac{1}{x}]$. Then it is clear that we have a unique surjection $\psi : S \twoheadrightarrow K$ sending $x$ to $\pi$ which is continuous in the $x$-adic topology, namely the quotient map with respect to the prime $(p) \subset S$. Choose any power series $\tilde{\alpha}_1, ... , \tilde{\alpha}_d \in W(\bar{\ff}_p)[[x]] \subset S$ whose coefficients all lie in $W(\bar{\ff}_p)^{\times}$ and which satisfy $\phi(\tilde{\alpha}_i) = \alpha_i$ for $1 \leq i \leq d$; note that we already have $e_{i,j} = E_{i,j}$ for $1 \leq i < j \leq d$. Then we get the statement of Lemma \ref{lemma etale divisor} through a very similar argument, which in turn allows us to prove Proposition \ref{prop reducing to K} exactly as we did for characteristic $0$.

\end{rmk}

\subsection{Proof of Theorem \ref{thm comparison}(b)} \label{sec4.4}

By what has been shown in the above subsections, along with Remark \ref{rmk intersection behavior}(a), we may assume with loss of generality that $R=\cc[[x]]$ and that $a_i = \alpha_i \in \cc[x]$ for $1 \leq i \leq d$.  The statement of Theorem \ref{thm comparison}(b) now results from the following proposition. (For a sketch of an alternative approach, see the end of \cite[\S2.3]{wewsy}. Note that Remark 2.13 in the same paper does not imply a description of the non-equicharacteristic inertia action, because \cite[Theorem 2.12]{wewsy} and its proof are insufficient for determining this action.)

\begin{prop} \label{prop isomorphic to family over cc((x))}

Theorem \ref{thm comparison}(b) holds when $R=\mathbb{C}[[x]]$ and $\alpha_i = a_i \in \cc[x]$ for $1 \leq i \leq d$.

\end{prop}

\begin{proof}

Let $X=\aff_{\cc}^1 \smallsetminus (\{0\}\cup \{z \ | \ a_i(z)=a_j(z)$ for some $i\neq j \})$. Let $\mathcal{F}' \to X$ be the family given by $$\mathcal{F}' = (\proj_{\cc}^{1} \times X) \smallsetminus \bigcup_{i = 1}^{d} \{(a_{i}(z), z) \ | \ z \in X\}.$$
 Then $(\mathcal{F}')^{\an} \to X^{\an}$ is clearly a fiber bundle which has $\mathcal{F} \to B_{\varepsilon}^{*}$ as a sub-bundle.  Fix an algebraic closure $\overline{ \cc((x))}$ of $\cc((x))$, and let $\bar \eta:\oper{Spec}(\overline{ \cc((x))})\rightarrow X$ be the corresponding geometric point lying over the generic point of $X$. The induced morphism  $\widehat{\zz} \cong G_{\cc((x))} \to \pi_{1}^{\et}(X, \bar \eta)$ is an  injection since every nontrivial element of $\pi_{1}^{\et}(X, \bar \eta)$ has infinite order. Similarly, the inclusion $B_{\varepsilon}^{*} \subset X^{\an}$  induces a map $\widehat{\zz} \cong \widehat{\pi}_{1}(B_{\varepsilon}^{*}, z_{0}) \to \pi_{1}^{\et}(X, z_{0})$ which is also an injection.

A path between $z_0$ and $\bar \eta$ induces an isomorphism of $\pi_{1}^{\et}(X, z_{0})$ with $\pi_{1}^{\et}(X, \bar \eta)$. We remark that such a path can be chosen so that the images of both $G_{\cc((x))}$ and $\widehat{\pi}_{1}(B_{\varepsilon}^{*}, z_{0})$ in $\pi_{1}^{\et}(X, z_{0})$ are equal. Indeed, if $Y = X \cup \{0\} \subset \cc$, then the kernel of $\pi_{1}^{\et}(X, z_{0})\rightarrow\pi_{1}^{\et}(Y, z_{0})$ is topologically generated by the conjugates of the image of $\widehat{\pi}_{1}(B_{\varepsilon}^{*}, z_{0})$; and for any path, the induced image of $G_{\cc((x))}$ in $\pi_{1}^{\et}(X, z_{0})$ is contained in this kernel. Since varying the path conjugates the image of $G_{\cc((x))}$, there exists a path so that its image is contained in the image of $\widehat{\pi}_{1}(B_{\varepsilon}^{*}, z_{0})$. To show that for such a path the images are equal, it suffices to check their restrictions with respect to a cover of $X$ of the form $y^{n} = x$, where it is clear.
   
Let $\infty_{\bar \eta}$ in $\mathcal{F}'_{\bar \eta}$ be the $\overline{\cc((x))}$-point lying over the point at infinity. Then it is easy to see that there exists a path between $\infty_{z_0}$ and $\infty_{\bar \eta}$ that induces an isomorphism $\pi_{1}^{\et}(\mathcal{F}', \infty_{\bar \eta}) \stackrel{\sim}{\to} \pi_{1}^{\et}(\mathcal{F}', \infty_{z_{0}})$ that commutes with the isomorphism $\pi_{1}^{\et}(X, \bar \eta) \stackrel{\sim}{\to} \pi_{1}^{\et}(X, z_{0})$ above, as well as with the infinity sections.  By \cite[Corollaire X.2.4]{grothendieck224revetements}, there exists a ``specialization map" 
$$\mathrm{sp}: \pi_{1}^{\et}(\mathcal{F}'_{\bar \eta}, \infty_{\bar \eta}) \twoheadrightarrow \pi_{1}^{\et}(\mathcal{F}'_{z_{0}}, \infty),$$
which Grothendieck's Specialization Theorem (\cite[Corollaire X.3.9]{grothendieck224revetements}) says is an isomorphism.  We obtain the following commutative diagram of short exact sequences of groups.

\begin{equation} \label{commutative diagram cc((x))} \xymatrix{ 1 \ar[r] & \widehat{\pi}_{1}(\mathcal{F}_{z_{0}}, \infty_{z_0}) \ar[r] \ar@{=}[d] & \widehat{\pi}_{1}(\mathcal{F}, \infty_{z_{0}}) \ar[r] \ar@{^{(}->}[d] & \widehat{\pi}_{1}(B_{\varepsilon}^{*}, z_{0}) \ar@/_1pc/[l] \ar[r] \ar@{^{(}->}[d] & 1
\\ 1 \ar[r] & \pi_{1}^{\et}(\mathcal{F}'_{z_{0}}, \infty_{z_0}) \ar[r] & \pi_{1}^{\et}(\mathcal{F}', \infty_{z_{0}}) \ar[r] & \pi_{1}^{\et}(X, z_{0}) \ar@/_1pc/[l] \ar[r] & 1 
\\ 1 \ar[r] & \pi_{1}^{\et}(\mathcal{F}'_{\bar \eta}, \infty_{\bar \eta}) \ar[r] \ar[u]^{\mathrm{sp} \wr} & \pi_{1}^{\et}(\mathcal{F}', \infty_{\bar \eta}) \ar[r] \ar[u]^{\wr} & \pi_{1}^{\et}(X, \bar \eta) \ar@/_1pc/[l] \ar[r] \ar[u]^{\wr} & 1
\\ 1 \ar[r] & \pi_{1}^{\et}(\mathcal{F}'_{\bar \eta}, \infty_{\bar \eta}) \ar[r] \ar[u]^{\wr} & \pi_{1}^{\et}(\mathcal{F}'_{\cc((x))}, \infty_{\cc((x))}) \ar[r] \ar@{^{(}->}[u] & G_{\cc((x))} \ar@/_1pc/[l] \ar[r] \ar@{^{(}->}[u] & 1 } \end{equation}

It is easy to check from the construction of $\mathrm{sp}$ in \cite[Expos\'{e} X]{grothendieck224revetements} that the middle two squares commute.  The top row is obtained by applying the right-exact functor of taking profinite completions to the short exact sequence (\ref{short exact sequence of topological fundamental groups}) of (topological) fundamental groups.  The fact that the bottom three rows are short exact sequences follows from \cite[Proposition XIII.4.3 and Exemples XIII.4.4]{grothendieck224revetements}.  The left-exactness of the top row follows immediately from a diagram chase; it is obvious that $\rho_{\tp}$ is induced by the splitting of this row.  The top and bottom maps shown in the middle column are injections by the Five Lemma.

The commutativity of the above diagram shows that the top and bottom rows are isomorphic as split short exact sequences and that therefore the induced actions are isomorphic, as desired.

\end{proof}

\begin{rmk} \label{rmk semidirect}

We observe that it is implicit in the proofs of Propositions \ref{prop reducing to K} and \ref{prop isomorphic to family over cc((x))} that if the $\alpha_i$'s satisfy the additional hypothesis in the second statement of Theorem \ref{thm comparison}(b), we have $\widehat{\pi}_1(\mathcal{F},\infty_{z_0})^{(p')} \cong \pi_1^{\et}(\mathbb{P}^1_K\smallsetminus \{\alpha_1,...,\alpha_d\},P)^{(p')}$.

\end{rmk}

\begin{rmk} \rm \label{rmk generating loops}
 
We make the identification $\pi_{1}^{\et}(\proj_{\bar{K}}^{1} \smallsetminus \{\alpha_{1}, ... , \alpha_{d}\}, \infty)^{(p')}$ with $\widehat{\pi}_{1}((\proj_{\cc}^{1})^{\an} \smallsetminus \{\alpha_{1}, ... , \alpha_{d}\}, \infty)^{(p')}$ via Riemann's Existence Theorem and the embedding $\bar{K} \hookrightarrow \cc$ fixed in the discussion leading up to Lemma \ref{lemma etale divisor}.  Then it is clear from the proofs of Propositions \ref{prop isomorphic to family over cc((x))} and \ref{prop reducing to K} that we can say a bit more about the isomorphism $\widehat{\pi}_{1}(\mathcal{F}_{z_{0}}, \infty_{z_{0}})^{(p')} \stackrel{\sim}{\to} \widehat{\pi}_{1}((\proj_{\cc}^{1})^{\an} \smallsetminus \{\alpha_{1}, ... , \alpha_{d}\}, \infty)^{(p')}$ given in the statement of Theorem \ref{thm comparison} (which is essentially a composition of maps coming from specializations and inclusions of algebraically closed fields).  Namely, this isomorphism takes the topological fundamental group of $(\proj_{\cc}^{1})^{\an} \smallsetminus \{a_{1}(z_{0}), ... , a_{d}(z_{0})\}$ to the topological fundamental group of $(\proj_{\cc}^{1})^{\an} \smallsetminus \{\alpha_{1}, ... , \alpha_{d}\}$ for $1 \leq i \leq d$.  Moreover, if $\sigma \in \widehat{\pi}_{1}(\mathcal{F}_{z_{0}}, \infty_{z_{0}})^{(p')}$ is the image of some element of $G_{\cc(x)}$ which lies in the conjugacy class of generators of inertia at the prime $(x - a_{i}(z_{0})) \in \Spec(\cc[x])$, then this isomorphism takes $\sigma$ to the image in $\widehat{\pi}_{1}((\proj_{\cc}^{1})^{\an} \smallsetminus \{\alpha_{1}, ... , \alpha_{d}\}, \infty)^{(p')}$ of an element $G_{\cc(x)}$ which lies in the conjugacy class of generators of inertia at the prime $(x - \alpha_{i}) \in \Spec(\cc[x])$.  If such a $\sigma$ actually lies in the topological fundamental group $\pi_{1}(\mathcal{F}_{z_{0}}, \infty_{z_{0}})$, then $\sigma$ may be viewed as a counterclockwise loop in $(\proj_{\cc}^{1})^{\an} \smallsetminus \{a_{1}(z_{0}), ... , a_{d}(z_{0})\}$ winding around only the missing point $a_{i}(z_{0})$, and the analogous statement is true for the image of $\sigma$ in $\pi_{1}((\proj_{\cc}^{1})^{\an} \smallsetminus \{\alpha_{1}, ... , \alpha_{d}\}, \infty)$ (see \cite[Remark 5.10]{volklein1996groups}).  In this way, we see that the above isomorphism of prime-to-$p$ \'{e}tale fundamental groups takes images of loops around the point $a_{i}(z_{0})$ to images of loops around the point $\alpha_{i}$ for each $i$.

\end{rmk}

\bibliographystyle{alpha}
\bibliography{bibfile2}

\end{document}